\documentclass[11pt,a4paper,reqno]{amsart}
\usepackage[margin=1.2in]{geometry}
\usepackage{graphicx,pdfsync,amssymb,mathrsfs,amsfonts,amsbsy,subcaption,color,esint,mathtools,tikz,mdframed}
\usetikzlibrary{positioning}
\mathtoolsset{showonlyrefs}

\usepackage[shortlabels]{enumitem}

\usepackage{hyperref}
\hypersetup{
colorlinks=true,
linkcolor=blue,
filecolor=blue,
urlcolor=blue,
citecolor=blue,
}

\usepackage[utf8]{inputenc}
\usepackage[T1]{fontenc}

\DeclareMathOperator{\Supp }{supp}
\DeclareMathOperator{\Id }{Id}

\DeclareMathOperator{\D}{div}

\newtheorem{theorem}{Theorem}[section]
\newtheorem{lemma}[theorem]{Lemma}
\newtheorem{proposition}[theorem]{Proposition}

\newtheorem{remark}[theorem]{Remark}

\setlist[enumerate]{itemsep=3pt}
\setlist[itemize]{itemsep=3pt}

\def \TT  {\mathbb{T}} 
\def \RR {\mathbb{R}}  
\def \NN {\mathbb{N}}  
\def \ZZ {\mathbb{Z}}  

\def \p {\partial}

\def \l {\lambda}

\def \ep {\epsilon}

\def \ez {\mathbf{e}_z}
\def \et {\mathbf{e}_\theta}

\def \er {\mathbf{e}_r}

\numberwithin{equation}{section}

\begin{document}

\title[Norm inflation  in supercritical   spaces]{Sharp norm inflation  for 3D Navier-Stokes equations in supercritical  spaces}

\author{Xiaoyutao Luo}

\address{State Key Laboratory of Mathematical Sciences, Academy of Mathematics and Systems Science, Chinese Academy of Sciences, Beijing  100190, China}

\email{xiaoyutao.luo@amss.ac.cn}


\date{\today}

\begin{abstract}
We prove that the incompressible Navier-Stokes equations exhibit norm inflation in $\dot B^{s}_{p,q}(\RR^3)$ with smooth, compactly supported initial data. Such norm inflation is shown  in all supercritical $\dot B^{s}_{p,q} $ near the scaling-critical line $s = -1+ \frac{3}{p}$ except at $s=0$. The growth mechanism differs depending on the sign of the regularity index $s$:  forward energy cascade driven by mixing for  $s>0$ and backward energy cascade caused by un-mixing for $s<0$. The construction also demonstrates arbitrarily large, finite-time growth of the vorticity, the first of such examples  for  the Navier-Stokes equations.
\end{abstract}

\date{\today}

\maketitle

\section{Introduction}\label{sec:intro}

In this paper, we consider  the three-dimensional  Navier-Stokes equations    of   incompressible    viscous fluids, 
\begin{equation}\label{eq:NS}
\begin{cases}
\p_t u -\Delta u + u\cdot \nabla u + \nabla p = 0 &\\
\D u =0  & \\
u|_{t = 0} = u_ 0 &
\end{cases}
(t,x) \in [0,T] \times \RR^3
\end{equation}
where $ u: [0,T] \times \RR^3  \to \RR^3$ is the unknown velocity field, $ p: [0,T] \times \RR^3 \to \RR $ is the scalar pressure, and $ u_0 : \RR^3 \to \RR^3$ is the initial data.

Our focus is the well-posedness/ill-posedness of the Cauchy problem \eqref{eq:NS}, where the notion of criticality plays a central role. Recall that a Banach space $X $ is called critical for \eqref{eq:NS} if its norm $ |\cdot |_X$ is invariant under the scaling
\begin{equation}\label{eq:intro_scaling_NS} 
u(t,x) \mapsto u_\l (t,x) : = \l u(\l^2 t,\l x )  .
\end{equation}
Examples of critical spaces include  $ L^3$, $\dot H^{\frac{1}{2}}$ and more generally, $ \dot B^{-1 +  {3}/{p}}_{p, \infty }$. Heuristically, \eqref{eq:NS} exhibits local well-posedness in critical or subcritical spaces but might become ill-behaved in supercritical regimes where scaling favors nonlinearity over dissipation.

The well-posedness/ill-posedness of \eqref{eq:NS} in supercritical spaces poses significant challenges.  While there have been \emph{existence} results in supercritical spaces, from Leray's seminal weak solutions~\cite{MR1555394} to Calder\'on's splitting \cite{MR968416,2410.07816}, these solutions are not known to preserve their initial data's regularity, except in the energy space. Such a   limitation motivates fundamental questions about the persistence of regularity and stability in supercritical settings.

In \cite{2404.07813} it was proved that \eqref{eq:NS} exhibits norm inflation in $ H^{s} $ when $ 0< s  < \frac{1}{2}$, i.e. solution with arbitrarily small initial data can be arbitrarily large in an arbitrarily short time.  The main result of this paper establishes norm inflation for \eqref{eq:NS} in \emph{almost all} supercritical spaces near the critical regularity threshold:

\begin{theorem}\label{thm:Besov}
For any $s \neq 0$ and $1\leq p,q \leq \infty$ such that $-3 < s -   \frac{3}{p} < -1 $, the 3D Navier-Stokes equations \eqref{eq:NS} are strongly ill-posed in $\dot B^{s}_{p,q} (\RR^3)$  in the following sense.

For any $\ep>0$, there exists a time $0< t^* \leq  \ep   $ and a   solution $u$ of \eqref{eq:NS} such that the following holds.

\begin{enumerate}

\item  The solution $u$ is smooth on $[0,t^*]$, i.e. 
$$
u \in C^\infty( [0,t^*] \times \RR^3 ) .
$$

\item The initial data $u |_{t= 0 } =u_0    \in C^\infty_c(   \RR^3)$   and
\begin{equation}\label{eq:thm_Hs_NS_1}
| u_0 |_{ \dot {B}^{s}_{p , 1 }(\RR^3) } \leq \ep.
\end{equation}

\item  $u$ develops norm inflation   at $t= t^* $:
\begin{equation}\label{eq:thm_Hs_NS_2}
| u   (  t^*  ) |_{\dot{B}^{s}_{p, \infty  }(\RR^3) } \geq \ep^{-1} .
\end{equation}

\end{enumerate}

\end{theorem}

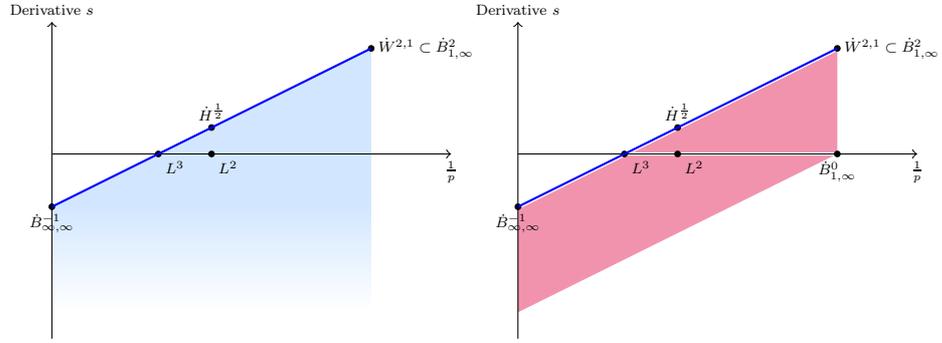
\begin{figure}[ht]
\centering 
\begin{subfigure}[b]{0.3\textwidth}
\begin{tikzpicture}[scale=0.70, every node/.style={transform shape}]

\fill[fill={rgb, 255:red, 210; green, 231; blue, 255 }  ,fill opacity= 1 ] (0,-1)  -- (0,-2) .. controls  (4,-2)  ..  (6,-1.5) -- (6, 2);
{  \usepgflibrary {shadings}
\fill[top color={rgb, 255:red, 210; green, 231; blue, 255 }] (0,-3) rectangle (6,-1); }

\draw[->] (0,0) -- (7.5,0) node[anchor=north] {\scriptsize   $\frac{1}{p}$};
\draw[->] (0,-3.5) -- (0,2.5) node[anchor=south ] {\scriptsize Derivative $s$};

\filldraw	(0,-1) circle (1.5pt) node[anchor=north] {\scriptsize  $\dot B^{- 1}_{\infty , \infty } $} ;
\filldraw (2,0) circle (1.5pt) node[anchor=north west] {\scriptsize $   L^{3}   $} ;
\filldraw (3,0) circle (1.5pt) node[anchor=north west] {\scriptsize $   L^{2}   $} ;
\filldraw (3,0.5) circle (1.5pt) node[anchor=south ] {\scriptsize $   \dot H^{\frac{1}{2}}   $} ;


\filldraw	 
(6,2) circle (1.5pt) node[anchor=west] {\scriptsize $ \dot W^{2, 1} \subset \dot B^{2}_{1 , \infty } $};


\draw[thick,blue] (0,-1) -- (6,2);

\end{tikzpicture}
\end{subfigure}
\qquad\qquad
\begin{subfigure}[b]{0.3\textwidth}
\begin{tikzpicture}[scale=0.70, every node/.style={transform shape}]

\fill[fill={rgb, 95:red, 90; green,  55; blue,  65 }  ,fill opacity= 0.6 ] (0.00,-1)   -- (0.00,-3) --  (6,-0) -- (6, -0.03)-- (2.04,-0.03)--  (0.1,-1  ) ;

\fill[fill={rgb, 95:red, 90; green,  55; blue,  65 }  ,fill opacity=  0.6 ] (2+0.1, 0.03)  -- (6,  0.03 ) --  (6, 2 - 0.05 ) -- (2.16 , 0.03)  ;

\draw[->] (0,0) -- (7.5,0) node[anchor=north] {\scriptsize   $\frac{1}{p}$};
\draw[->] (0,-3.5) -- (0,2.5) node[anchor=south ] {\scriptsize Derivative $s$};

\filldraw	(0,-1) circle (1.5pt) node[anchor=north] {\scriptsize  $\dot B^{- 1}_{\infty , \infty } $} ;
\filldraw (2,0) circle (1.5pt) node[anchor=north west] {\scriptsize $   L^{3}   $} ;
\filldraw (3,0) circle (1.5pt) node[anchor=north west] {\scriptsize $   L^{2}   $} ;
\filldraw (3,0.5) circle (1.5pt) node[anchor=south ] {\scriptsize $   \dot H^{\frac{1}{2}}   $} ;

\filldraw	 
(6,2) circle (1.5pt) node[anchor=west] {\scriptsize $ \dot W^{2, 1} \subset \dot B^{2}_{1 , \infty } $};


\filldraw	 
(6 ,-0) circle (1.5pt) node[anchor= north] {\scriptsize $     \dot B^{  0 }_{1 , \infty } $};

\draw[thick,blue] (0,-1) -- (6,2);

\end{tikzpicture}
\end{subfigure} 
\caption{Left: Supercritical region in light blue. Right:   Ill-posedness region of Theorem \ref{thm:Besov} in pink}
\label{fig:ill-posedness}
\end{figure}

As depicted in Figure \ref{fig:ill-posedness}, Theorem \ref{thm:Besov} establishes norm inflation for \eqref{eq:NS} in $ \dot B^{s}_{p,q}$ near the critical line $s = -1+ \frac{3 }{p}$ except possibly at $s=0$. It remains an open question whether such norm inflation can occur in supercritical Lebesgue spaces $L^p$, $2< p< 3$.

\begin{remark}
\hfill

\begin{itemize}

\item  The upper limit $s  = -1+ \frac{3}{p}$ is sharp---the solution map is continuous for small data in $B^{-1 + 3/p }_{p, \infty}$, $p<\infty$, see for instance~\cite[Theorems 5.40]{MR2768550}.

\item Theorem \ref{thm:Besov} extends the famous $\dot B^{-1}_{\infty, \infty}$ norm inflation of Bourgain-Pavlovi\'{c}    to  $\dot B^{-s}_{\infty, \infty}$, $ -3 < s < -1$.

\item  While the lower limit $s  = -3+ \frac{3}{p}$ appears to be technical, the result covers most of the practical cases. In fact, $L^1 \not\hookrightarrow \dot B^{s}_{p, q }$ for $s  < -3+ \frac{3}{p}$.

\item By standard embedding results, Theorem \ref{thm:Besov} implies the same strong ill-posedness in Sobolev spaces $\dot W^{s,p}$  as well.

\end{itemize}
\end{remark}

As said in the abstract, the driving mechanism of norm inflation differs between $s>0 $ and $s<0$. The growth in the case $s>0$ is based on mixing, producing a forward energy cascade. On the contrary, for $s<0$, the growth of negative norms results from a backward energy cascade caused by un-mixing. Therefore the restriction of $s \neq 0$ is essential in our construction. 
\subsection{Application to small-scale formation}

A key implication of the norm inflation in the $s >0$
case is its connection to the small-scale formation of \eqref{eq:NS}. In potential blowup scenarios, the transfer of energy to finer scales can lead to singularities. To quantify this, one may examine the growth ratio of subcritical norms, such as:
\begin{equation}\label{eq:intro_growth}
\frac{| \omega  (t )|_{L^\infty } }{|\omega_0 |_{L^\infty }} 
\end{equation}
where $\omega$ denotes the vorticity of the solution. By the Beale-Kato-Majda criterion~\cite{MR763762}, unbounded growth of \eqref{eq:intro_growth} is a sufficient and necessary condition of a finite-time blowup. Recent numerical investigations~\cite{Hou23} found a class of pure swirl initial data that exhibits significant sustained growth of \eqref{eq:intro_growth}, providing empirical support for potential finite-time blowup.

Despite recent developments on blowup scenarios/small-scale formation in related fluid models~\cite{MR3245016,MR3359050,MR4334974,MR4255049,MR4527834,2410.22920}, no rigorous lower bounds for \eqref{eq:intro_growth} have been established for solutions of the Navier-Stokes equations to date.  In other words, there remains no proof of growth—let alone sustained growth or blowup—for solutions of \eqref{eq:NS}.

A direct corollary of our construction demonstrates  unbounded growth ratios within finite time for smooth solutions of \eqref{eq:NS}:

\begin{theorem}\label{thm:growth}
For any $M>0$, there exist  a time $  t^* >0   $ and a   smooth solution $u$ of \eqref{eq:NS} on $[0,t^*]$ such that its vorticity $\omega= \nabla \times u$ satisfies
\begin{equation}\label{eq:thm:growth}
\frac{| \omega (t^*)|_{L^\infty } }{| \omega_0 |_{L^\infty }}  \geq M .
\end{equation}

In particular, the initial data $u_0\in C^\infty_c(\RR^3)$ can be arbitrarily small in any prescribed supercritical (homogeneous) Sobolev/Besov space.

\end{theorem}

One can show a similar growth ratio  for other   norms such as $|u|_{W^{s,p}}$ with $s>0 $ in \eqref{eq:thm:growth}.   

The growth in Theorem \ref{thm:growth} is significant but mere for a  short burst of   time. To find potential blowup candidates, one needs to establish a feedback loop of self-sustained growth in \eqref{eq:NS}, which is far beyond the scope of the current paper.

\subsection{Background and literature}

We now discuss the related background of our result.
As the literature on well-posedness of \eqref{eq:NS} is too vast to give a complete account here, we only mention several key milestones and refer to the monographs  \cite{MR3469428,MR4475666} for further references.

\textbf{Weak and mild solutions.} For $L^2$ initial data, the existence of a global weak solution was established in Leray's seminal work~\cite{MR1555394}. Leray solutions are constructed via compactness arguments and satisfy the energy inequality. However, their regularity and uniqueness remain open problems. See recent \cite{Hou23} for numerics of potential blowup and~\cite{2112.03116} for non-uniqueness under external forcing.

An alternative framework, the mild solution approach  treats \eqref{eq:NS} as a perturbation of the heat equation and involves the equivalent integral formulation
\begin{equation}\label{eq:intro_mild}
u(t) = e^{\Delta t} u_0 - \int_0^t e^{\Delta (t -\tau )} \mathbb{P}\D(u\otimes u ) (\tau) \, d\tau 
\end{equation}
where $\mathbb{P} $ is the Leray projection.

In the framework of mild solution \eqref{eq:intro_mild}, the notion of criticality plays an important role. Under the scaling \eqref{eq:intro_scaling_NS},  the following is a well-known hierarchy of embedded critical spaces 
\begin{equation*}
\dot{H}^{\frac{1}{2}} \subset L^3 \subset \dot B^{-1 + \frac{3}{p}}_{p, \infty } \,\,, p<\infty   \subset  BMO^{-1} \subset  \dot B^{-1}_{\infty, \infty } .
\end{equation*}

In \cite{MR166499}, Fujita and Kato initiated the research on mild solutions, establishing local well-posedness of \eqref{eq:NS} for $\dot H^{\frac{1}{2}}$ initial data. 

Subsequent developments in $L^3$  \cite{MR760047,MR833416} and in Besov spaces \cite{Cannone1993-1994,MR1617394} culminated with Koch-Tataru's $BMO^{-1}$ theorem \cite{MR1808843} that remains the sharpest  local well-posedness of \eqref{eq:NS} without structural assumptions. See recent \cite{2503.14699} for the sharp non-uniqueness construction to the Koch-Tataru's $BMO^{-1}$ theorem.

\textbf{Ill-posedness in critical regimes.}
It had been conjectured~\cite{MR1724946,MR2099035} that \eqref{eq:NS} is ill-posed in $\dot B^{-1}_{\infty, \infty}$, the largest critical space. Indeed, the groundbreaking work of Bourgain-Pavlovi\'{c} \cite{MR2473255} confirmed this via a norm inflation construction. We also mention  other mild ill-posedness   results around the same time of~\cite{MR2473255}. In~\cite{MR2473256}, the author proved that the solution map  is not $C^2$ in $\dot B^{-1}_{\infty, q }$ for $q>2$, whereas in \cite{MR2566571}, the authors show a jump discontinuity in $\dot B^{-1}_{\infty, \infty }$. The result of Bourgain-Pavlovi\'{c}~\cite{MR2473255} was improved in subsequent work  \cite{MR2601621,MR3254526,MR3276597}, eventually showing ill-posedness in $\dot B^{-1}_{\infty, q } $ for $q>2$ in \cite{MR3254526} and all $ 1 \leq q\leq 2 $ in ~\cite{MR3276597}.

This scarcity of ill-posedness results contrasts starkly with the vast scope of well-posedness theory. Remarkably, in the critical regime, norm inflation phenomena are confined to endpoint spaces $p=\infty$--- the solution map is continuous for small data in $B^{-1 + 3/p }_{p, \infty}$, see for instance~\cite[Theorems 5.40]{MR2768550}.

\textbf{Supercritical regimes and beyond.} In supercritical spaces--where nonlinearity dominates--the classical well-posedness theory of mild solutions collapses. These spaces include:
\begin{equation*}
\dot{H}^{s} , (s < \tfrac{1}{2}), \quad L^p , (p < 3), \quad \dot B^{s}_{p,q} , \left(s < -1 + \tfrac{3}{p}\right).
\end{equation*}

Despite this, existence results via compactness method survive,  as in classical Leray weak solutions. Calderón~\cite{MR968416} pioneered a hybrid approach via solution splitting  $u(t) = u_{m}  + u_w $ : the mild part  $u_{m}$ is constructed à la Kato in critical spaces, while the weak part is a Leray-type weak solution of a perturbed system. This ``nonlinear interpolation'' approach was revisited in various critical settings~\cite{MR3614655,MR3902173,MR3916974}.

Calderón's splitting was   recently extended to supercritical Besov spaces $\dot B^{s}_{p,q} $, $ -1 +\tfrac{2}{p}< s< -1 +\tfrac{3}{p}$ by Popov~\cite{2410.07816}.  However, such solutions cannot maintain the initial data's regularity, as our norm inflation result demonstrates.

While controlling the solutions in supercritical regimes poses significant challenges, recent developments suggest that \eqref{eq:NS} should be ill-posed in such regimes. In particular, it was shown by convex integration that \eqref{eq:NS} admits nonunique \emph{weak solutions} under supercritical regularity~\cite{MR3898708,MR4462623}. However, those weak solutions exhibit characteristics much different from the smooth ones considered in the norm inflation literature. We also mention   recent~\cite{2503.14699}, where the authors combined convex integration with aspects of critical mild formulation to show nonuniqueness of large $BMO^{-1}$ solutions.

\textbf{Small scale formation and blowup.} In recent years, there have been significant advances in understanding blowup scenarios and small-scale formation in related fluid models~\cite{MR3289364,houluo14,MR3245016,MR3486169,MR4334974,MR4255049,MR4527834,2410.22920}. The growth mechanisms of small-scale formation are sometimes closely related to norm inflation constructions, as demonstrated by~\cite{MR3245016,MR3625192}

However, no growth mechanism had been rigorously established for solutions of \eqref{eq:NS}  prior to Theorem \ref{thm:growth}. This is in part due to the viscous dissipation prohibits or suppresses many   growth mechanisms known in the inviscid setting, including the DiPerna-Majda $2 \frac{1}{2}$D setup \cite{MR0877643}, the 3D Hou-Luo scenario \cite{houluo14,2210.07191},  Bourgain-Li's deformation of oscillation \cite{MR3359050}, and Elgindi's Euler self-similar  blowups \cite{MR4334974}. Furthermore, Leray's backward self-similar solutions, once considered prime candidates for singularity formulation, were ruled out in the 1990s~\cite{MR1397564,MR1643650}.

In~\cite{2407.06776}, the authors show singularity formation for a generalized Navier-Stokes system with fractional dissipation of small order and external forces in $L^1_t C^{1,\ep}$. We refer to~\cite{2309.08495,2407.06776,2410.22920} for a detailed discussion of the methodology.

As in other norm inflation constructions, our solutions exhibit only an instantaneous burst of growth. A central open question remains: Can solutions to \eqref{eq:NS} achieve prolonged, sustained growth or even finite time blowup?

\subsection{Outline of construction}

The argument of proving Theorem \ref{thm:Besov} is based on~\cite{2404.07813}. The essential steps consist of: 
\begin{enumerate}[label=(\alph*)]
\item Constructing a simplified approximate system capturing key dynamics of \eqref{eq:NS};
\item Demonstrating norm inflation  within this approximate system;
\item Establishing proximity between the approximate and exact solutions of \eqref{eq:NS}.
\end{enumerate}

This   general scheme of proving ill-posedness of fluid  equations  was already rooted in the pioneering work~\cite{MR3359050,MR3320889} but brought to systematic treatments in recent works~\cite{2107.07463,2210.17458}.

\textbf{Step (a): Approximations.}
To find the approximate system, we follow the two steps reductions as in~\cite{2404.07813}:
\begin{itemize}
\item \textit{3D Navier-Stokes to 3D Euler}: In supercritical spaces, viscous effects are dominated by the nonlinearity over short timescales, allowing \eqref{eq:NS} to approximate Euler dynamics.
\item \textit{3D Euler to 2D Euler + transport}: Exploiting axisymmetric anisotropy, we reduce the system into a 2D Euler equation (governing radial and vertical components) and a linear transport equation for the swirl component.
\end{itemize}

\textbf{Step (b): Norm inflation mechanism.}
For the reduced system of ``2D Euler + liner transport'', norm inflation arises solely from the swirl component’s evolution, while radial/axial components remain stationary.  

\begin{itemize}
\item \textit{Case $s>0$}: Non-Lipschitz velocity fields (permitted by supercriticality) induce rapid growth via  mixing. Our construction is based on \cite{2404.07813}, with enhanced bookkeeping for Besov estimates.
\item \textit{Case $s<0$}: Initial data is carefully ``premixed''---time evolution unmixes it, inducing a decay of positive regularity norms, which is turned into the growth of negative regularity norms.
\end{itemize}

\textbf{Step (c): Stability of the approximation.}

Once we establish the desired growth for the approximate solution,  critical to the proof is showing that norm inflation persists under the full \eqref{eq:NS} dynamics. 

This requires that our reduced system remain a good approximation of \eqref{eq:NS}. On one hand,   the supercritical regularity allows to neglect the viscous dissipation until the norm inflation occurs. On the other hand, the anisotropic flow structures provide the necessary smallness allowing to recover from ``2D Euler+linear transport'' to the full 3D Euler dynamics.

\subsection*{Organization of the paper}
The rest of the paper is structured as follows.
\begin{itemize}
\item Section~\ref{sec:pre} introduces  preliminaries of necessary functional settings and essential inequalities.

\item Section~\ref{sec:overline_u} constructs approximate solutions for the reduced system of ``2D Euler + transport''.

\item Sections~\ref{sec:overline_u_inflation_s>0} (for $s>0$) and~\ref{sec:overline_u_inflation_s<0} (for $s<0$) prove norm inflation in the approximate system.

\item Section~\ref{sec:no_blowup} compares approximate and exact solutions, establishing smoothness and stability up to the critical time $t^*$.

\item Section~\ref{sec:proof} combines these results to conclude Theorem~\ref{thm:Besov} and Theorem~\ref{thm:growth}.
\end{itemize}

\subsection*{Acknowledgment}
The author is partially supported by NSFC No. 12421001 and  No. 12288201.

\subsection{Notations}

For a vector- or tensor- value function $f$, its modulus $|f|$ denotes   the square root of the sum of squares of each component.

For a Banach space $X$, its norm is denoted by $|\cdot|_X$. Functional norms in this paper are mostly defined on $\RR^3$, so we often use $|\cdot |_{L^p }  $ and $ |\cdot |_{W^{k,p} }$ for brevity.

For two quantities $X,Y \geq 0$, we write $X \lesssim Y$ if $X \leq C Y$ holds for some constant $C >0$, and similarly $X \gtrsim Y$ if $X \geq C Y$, and $X \sim Y$ means $X \lesssim Y$ and $X \gtrsim Y $ at the same time.

In addition, $X \lesssim_{a,b,c,\dots} Y$ means $X \leq C_{a,b,c\dots} Y$ for a constant $C_{a,b,c\dots}$ depending on parameters $a,b,c\dots$.

Throughout the paper, for $k\in \NN$, $ \nabla^k $ refers to the full gradient in $\RR^3$. We denote the Fourier transform of a function $f$ by $\widehat{f}$.

\subsection{Sobolev and Besov spaces}

We recall the definition of Sobolev spaces for integer $ k \in \NN$,
\begin{equation}\label{eq:def_Wkp}
|f |_{   W^{k, p }  } =   \sum_{0 \leq  i \leq k } |\nabla^i f|_{L^p  } \qquad |f |_{   \dot W^{k , p }  }  =   |\nabla^k f|_{L^p  } .
\end{equation}

We also recall the fractional Sobolev spaces of the following definition. For real $s \in \RR $ and $1 <  p < \infty $, 
\begin{equation}\label{eq:def_Wsp}
|f  |_{  W^{s, p }}  =  | J^s f |_{L^p } \qquad  |f  |_{  \dot W^{s, p }}  =  | \Lambda^s f |_{L^p }
\end{equation}
where $J^s$ is the Bessel potential with the Fourier multiplies $ \widehat{J^s} = (1 + |\xi|^2)^{ \frac{s}{2}}$ and $\Lambda^s$ is the Riesz potential with the Fourier multiplies $ \widehat{\Lambda^s} =    |\xi|^s $.  When $ p =2$, we denote $  H^s: =   W^{s,2 } $ and $\dot H^s: =  \dot W^{s,2 } $. 

It is well-known that for $1< p< \infty$, when $s =k$ is an integer, both definitions \eqref{eq:def_Wkp} and \eqref{eq:def_Wsp} coincide. In this paper we only consider either $k\in \NN$ with $1 \leq p \leq \infty$ or $k\in \RR$ with $1 <  p < \infty$.

\subsection{Besov spaces}

We recall the definition of Besov spaces by the Littlewood-Paley decomposition. Since we only use some duality and interpolation properties of Besov spaces, we refer to \cite{MR2768550} for the details.

Let $\{ \Delta_q\}_{q \in \ZZ} $  be a sequence of Littlewood-Paley projection operators such that $\widehat{\Delta_q} $ is supported in frequencies $|\xi| \sim 2^{q}$ and $\Id = \sum_{ q\in \ZZ} \Delta_q $ in the sense of distribution.

For $s \in \RR$, $1 \leq  p, q \leq \infty$, the homogeneous Besov norm $\dot B^{s}_{p,q} $ is given by
\begin{equation}\label{eq:pre_besov_1}
|f|_{\dot B^{s}_{p,q} } = \Big| 2^{s q}| \Delta_q f |_{L^p }    \Big|_{\ell^q(\ZZ)}
\end{equation}

We also frequently the following fact:  
\begin{equation}\label{eq:besov_embedding}
\dot B^{k}_{p,1 } \hookrightarrow  \dot W^{k, p } \hookrightarrow \dot B^{k}_{p, \infty }
\end{equation}

We do not use the definition \eqref{eq:pre_besov_1} to actually compute Besov norms. Instead, we often use the following tools. The first is a convenient interpolation result.
\begin{lemma}\label{lemma:besov_interpolation}\cite[Proposition 2.22]{MR2768550}
For any $s_1 < s_2 $ and $ 0 < \alpha < 1$, let $s: = \alpha s_1 + (1 -\alpha ) s_2   $. Then
$$
|f |_{\dot{ B }^{  s }_{p,1 }} \lesssim_{\alpha,s_1,s_2} |f|_{\dot{ B }^{ s_1   }_{p, \infty  }}^\alpha   |f|_{\dot{ B }^{ s_2   }_{p, \infty  }}^{1-\alpha }.
$$
\end{lemma}

The second, valid for all $s\in \RR$ but especially useful for $s<0$, is the duality  
\begin{equation}\label{eq:besov_duality}
|  f |_{\dot B^{s}_{p,q }} \sim \sup_{  \phi \in Q^{- s  }_{p',q'}}     \int_{\RR^3} f \phi \, dx 
\end{equation}
where  for $1 \leq p,q \leq \infty $  $Q^{- s }_{p', q'} : = \{  \phi  \in  \mathcal{S} (\RR^3): |\phi  |_{\dot{ B }^{     -s  }_{p', q' }  }   \leq  1 \}$ with $ \mathcal{S}(\RR^3)$ the Schwartz class and $p',q' $ being the H\"older conjugates  of $p,q $.

Finally, our main tool is the following   Kato-Ponce commutator estimate.

\begin{proposition}\label{prop:kato_ponce}\cite[Propostion 4.2]{MR951744}
Let  $u$ be a smooth solution of \eqref{eq:NS} on $[0,t_0]$ for some $t_0>0$ such that $ |\nabla u   |_{L^\infty( [0,t_0] ;L^\infty) } \leq M$ for some constant $M \geq 1$.

Then for any   $k \geq 0$ and $1 < p< \infty$,
\begin{equation}\label{eq:prop_kato_ponce}
|u(t)|_{ {W}^{k,p}  }  \leq    |u_0 |_{  {W}^{k,p}  } e^{C_{k,p} M t } \quad \text{for all $t \in [0, t_0 ]$}
\end{equation}
for universal constants $C_{k,p} $ independent of $M$, $ t_0$, and $t$.
\end{proposition}

\section{The approximate solution}\label{sec:overline_u}

In this section, we construct a class of approximate solutions of \eqref{eq:NS}. In the construction, there are  two free parameters  whose values we fixed in Section \ref{sec:no_blowup} depending on the given $\ep>0$.

The main result is Proposition \ref{prop:approximate} below.  In the next two sections, we will show that the approximate solution $\overline{u}$ also stays close to the exact solution $u$ at least until $t^*$.

\begin{proposition}\label{prop:approximate}
Let $s,p $ be given as in Theorem \ref{thm:Besov}. For any $\ep>0$ and $\mu,\nu \geq 1$ satisfying \eqref{eq:def_nu_lambda}, there exists a smooth, divergence-free, axisymmetric vector field $\overline{u}: [0,\infty) \times \RR^3 \to \RR^3$ and a time $t^*(\mu,\nu,\ep) >0$ (defined in \eqref{eq:def_critical_t*}) such that the following holds.

\begin{enumerate}
\item \textbf{Regularity estimates:} $\overline{u} $ satisfies for any $k\in \NN$, the estimates 
\begin{align}\label{eq:prop_error_1}
| \overline{u}   |_{L^\infty ([0,t^*]; \dot W^{ k,{q} } )   }  \leq     C_{\ep,  k,{q}  }    \mu^{ k -s  } \mu^{ \frac{ 2}{p}  -\frac{ 2}{{q} }} \nu^{\frac{1}{p} - \frac{1}{{q}} }     .
\end{align}

\item \textbf{Approximation property:} There exist  a smooth autonomous pressure $\overline{p}:   \RR^3  \to  \RR      $  and   an error field $\overline{E}: \RR^+ \times  \RR^3  \to  \RR^3    $ such that 
\begin{equation}\label{eq:prop_error_2}
\begin{cases}
\p_t  \overline{u} -\Delta \overline{u} + \overline{u}\cdot \nabla \overline{u}   + \nabla \overline{p} = \overline{E} & \text{in $(t,x) \in \RR^+ \times  \RR^3  $} \\
\D \overline{u}  = 0 &\\
\overline{u} |_{t = 0 } = u_0  & 
\end{cases}
\end{equation}
where the error field $\overline{E}$, compactly supported in $\RR^3$ at each time $t \in \RR^+$, satisfy for any $k \in \NN $, the estimates 
\begin{equation}\label{eq:prop_error_3}
|  \overline{E}  |_{L^\infty ([0,t^*]; \dot W^{ k,{q} } )   }  \leq   C_{\ep,k , {q} }   \mu^{k - s }  (\mu^{-1} \nu)   (\mu^{1+\frac{ 2}{p}-  s}    \nu^\frac{1}{p}  )  \mu^{ \frac{ 2}{p}  -\frac{ 2}{{q} }} \nu^{\frac{1}{p} - \frac{1}{{q}} }  .
\end{equation}

\end{enumerate}

\end{proposition}

\subsection{Outline of construction}

We give a brief description of the construction of $\overline{u}$. As   discussed in the introduction, we first neglect the viscous effects and consider the axisymmetric Euler equations
\begin{equation}\label{eq:axisymmetric_euler}
\begin{cases}
\p_t  {  u}_{  \theta } +  {u}_{  r} \p_r  {u}_{  \theta } +  {u}_{  z} \p_z  {u}_{  \theta }+ \frac{1}{r}  {u}_{  \theta } {u}_{  r}  = 0 & \\
\p_t  {u}_{  r } +  {u}_{  r} \p_r  {u}_{  r } +  {u}_{  z} \p_z  {u}_{  r } - \frac{1}{r}  {u}_{  \theta }^2 + \p_r  {p} = 0& \\
\p_t  {u}_{  z} +  {u}_{  r} \p_r  {u}_{ z} +  {u}_{  z} \p_z  {u}_{  z}  +  \p_z    {p}  = 0 .&  
\end{cases}
\end{equation}
In \eqref{eq:axisymmetric_euler}, the nonlinear terms decompose into a two dimensional transport part of ${u}_{  r}, {u}_{ z} $ and $1/r$ terms. In our anisotropic setup (see Figure \ref{fig:schematic}), $1/r$ is much less than a full derivative, and hence $1/r$ terms are of lower order among the nonlinear terms. Dropping all $1/r$ terms, we then observe that 
\begin{itemize}
\item A swirl component  $u_\theta$ governed by a passive transport equation.
\item Radial and vertical components $(u_r, u_z) $ satisfying the 2D Euler equations in the $ r z $-plane.
\end{itemize}

There is a lot of freedom in constructing solutions to this reduced system. Specifically we take the following.
\begin{itemize}
\item The radial and vertical components $(u_r, u_z) $  to be a stationary radial vortex   in the $ r z $-plane.

\item The swirl $u_\theta $ supported inside the radial vortex and   passively transported  by $(u_r, u_z) $ in the  $ r z $-plane .
\end{itemize}

\subsection{Auxiliary coordinates}
Now we start the construction. We work in the cylindrical coordinate centered at the origin with $(x_1, x_2 ,x_3) \mapsto (\theta,r,z )$ defined by
\begin{equation}
x_1 = r \cos \theta , \quad x_2 = r \sin \theta ,\quad x_3 =z 
\end{equation}
and the associated orthonormal frame
\begin{equation}
\et = (   - \sin \theta ,  \cos \theta, 0 )  \quad \er = (   \cos \theta  ,  \sin \theta , 0) \quad  \ez = (0,0,1) .
\end{equation}

We note the point-wise bounds $ | \nabla^k \er | \lesssim r^{-k}  $ and $ | \nabla^k \et | \lesssim r^{-k}  $ for $k\in \NN$.

For any vector field $v:\RR^3\to  \RR^3$, we denote its cylindrical components by $ v = v_\theta \et  + v_r \er +  v_z \ez $. We say a function $f :\RR^3\to  \RR $ is axisymmetric if it does not depend on $\theta$, and similarly for a vector-valued function $v :\RR^3\to  \RR^3$  if its components $ v_\theta, v_r ,v_z $ are axisymmetric.

By an induction argument and the identity $\nabla f = \p_r f \er + \p_z f \ez $,  for any axisymmetric $f: \RR^3 \to \RR $ we have the following standard point-wise bounds for $k\in \NN$
\begin{equation}\label{eq:diff_r_z}
|\nabla^k f (x)|\lesssim_k \sum_{ 0\leq i  \leq  k}   | \p_r^{ i} \p_z^{k-i} f (x)|.
\end{equation}

We also work with  a shifted polar coordinate in the $rz$-plane (see Figure \ref{fig:schematic})
\begin{equation}\label{eq:def_rho_varphi}
r = \nu^{-1}  + \rho \cos(\varphi ), \quad z = \rho \sin(\varphi ).
\end{equation}
centered at the point $(r,z) = (\nu^{-1}  ,0  ) $ for some   $\nu\geq 1 $, a large parameter whose value we will fix in Section \ref{sec:no_blowup} depends on $\ep>0$, see also \eqref{eq:def_nu_lambda} below.

We also note the following useful point-wise bounds: if $f :\RR^3 \to \RR  $ is axisymmetric, then for $ k \in \NN$
\begin{equation}\label{eq:diff_rho_varphi}
|\nabla^k f (x)|\lesssim_k \sum_{ 0 \leq i +j \leq  k} \frac{1}{\rho^{k - i   }} | \p_\rho^{ i} \p_\varphi^j f (x)| \quad \text{when $ \rho  >0$}
\end{equation}
which can be proved by passing first to \eqref{eq:diff_r_z} and then use induction in the $rz$-plane. Another useful inequality, valid for any $f:\RR^3 \to \RR $ is
\begin{equation}\label{eq:diff_rho_varphi_2}
| \p_\rho^{ k}   f (x)| \leq   |\nabla^k_{r,z}   f (x)| \leq |\nabla^k f (x)|  \quad \text{when $ \rho  >0$}
\end{equation}
where $\nabla^k_{r,z} $ denotes taking the gradient in $rz$-plane.

We use the convention that if a function $f$ is defined by variables $r,z,\rho,\varphi $, we use the same letter $f$ to indicate its Euclidean counterpart $\RR^3\to \RR$ defined implicitly by these variables and vice versa.

It is worth emphasizing that all the norms $L^p$ or $W^{k,p}$ appearing below are taken in the original Euclidean variable of $\RR^3$, and we only use variables $(\theta,r, z)$ and $(\rho , \varphi )$ to simply the notations.

\subsection{The parameters and the setup}\label{subsec:setup}

In the construction of the  approximate solution $\overline u $ there are a few parameters. The first group consists of exponent parameters that depend on the input exponent $s, p$ from Theorem \ref{thm:Besov}. 

Throughout the paper we fix
\begin{itemize}
\item $ 0< b \ll 1$ given by
\begin{equation}\label{eq:def_small_b}
b = \frac{1}{10} \Big( - 1 - s + \frac{ 3}{p}  \Big).
\end{equation}

\item $N \gg 1 $  given by
\begin{equation}\label{eq:def_large_N}
N = \max\{\frac{100}{s},100 \} .
\end{equation}
\item $k_0 =6 $ such that $k_0  > |s|+3$, to facilitate the constructions when $s<0$.

\end{itemize}

The two major frequency parameters $\mu, \nu  \geq 1$  related by
\begin{equation}\label{eq:def_nu_lambda} 
\nu   = \mu^{1- b }   .
\end{equation}
The  parameter $\mu  $ represents the magnitude of a   derivative, while   $\nu^{-1}$ is the scale of the distance to the origin. Their values are fixed till the very end of the proof depending on $\ep>0$ and other universal constants. 

Define the critical time $t^* > 0$, the onset of norm inflation:
\begin{equation}\label{eq:def_critical_t*}
t^* =  \ep^{- N -2   }  \mu^{ - 1  -\frac{2}{p} + s } \nu^{-\frac{1}{p}} ,
\end{equation}
where $ N >0 $ is the large exponent that we fixed in \eqref{eq:def_large_N}. Note that  $  t^* \to 0 $ as $\mu \to \infty$ due to \eqref{eq:def_small_b} and \eqref{eq:def_nu_lambda}.

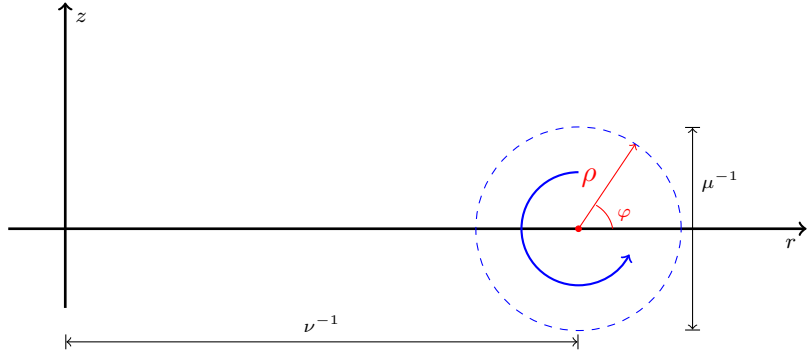
\begin{figure}[ht]
\centering
\begin{tikzpicture}[scale=1.5]

\draw[->][line width=1]  (-1.5,-0.7) -- (-1.5,2);
\draw  (-1.5,2) node[anchor=north west] {\scriptsize  $z$};
\draw[->][line width=1] (-2, 0) -- (5,0);
\draw (5.0,0)  node[anchor=north east] {\scriptsize $r$};

\draw[->][color= red  ] (3,0) -- (3.5,0.75);
\draw  (3.25,0.45) node[anchor= east] {  $\color{red} \rho$  };

\filldraw[color= red  ]  (3,0)  circle (0.7pt); 

\draw[|<->|] (-1.5,-1)--++ (4.5,0) node[pos=0.5,above]{\tiny $\nu^{-1}$};

\draw[|<->|] (4,-0.9)--++ (0,1.8); 

\draw (4.5, 0.6) node[anchor=north east]{\tiny $\mu^{-1}$};

\draw[blue  ,dashed] (3,0) circle (0.9cm);

\begin{scope}[shift={(3,0)}]  
\draw[->, thick, blue] (-270:0.5cm) arc (-270:-27:0.5cm);
\end{scope}

\draw[color= red  ] ( 3.3,0) arc(10:60:0.3) node[pos=0.5,right]{\tiny $\varphi$};

\end{tikzpicture}
\caption{Anisotropic setup in  $rz$-plane with  $\mu^{-1} \ll \nu^{-1}$. }
\label{fig:schematic}
\end{figure}

Throughout the construction, let us fix the profiles $f,g \in C^\infty_c(\RR )$  (in the shifted polar coordinate $(\rho, \varphi)$, see Figure \ref{fig:schematic}) such that
\begin{equation}\label{eq:def_profiles}
\begin{cases}
f ' (   \rho  )   =   1 \quad \text{for $ 1\leq \rho  \leq \frac{3}{2}$} & \\
f(\rho) = \frac{d^{N}}{d \rho^{N}} \widetilde{f}(\rho) \quad \text{for some $ \widetilde{f} \in  C^\infty_c(\RR )$} & \\
\Supp f(\rho)   \subset \{  \frac{1}{2}\leq \rho  \leq 2 \}    & \\
\Supp g(\rho)  \subset \{ 1\leq \rho  \leq \frac{3}{2} \} .
\end{cases}
\end{equation}

These profiles will be used to construct various components of the approximate solution $\overline{u}$.

\subsection{Definition of \texorpdfstring{$\overline{u}_r$}{ur} and \texorpdfstring{$\overline{u}_z$}{uz}}
Using the large parameters $\mu,\nu  \geq 1 $ and profiles $f,g$, we define the   approximate solution $\overline{u} :\RR^+ \times \RR^3 \to \RR^3 $ in cylindrical 
\begin{equation}\label{eq:def_overline_u}
\begin{aligned}
\overline{u}(t,x) &=   \overline{u}_\theta(t,x)  \et + \overline{u}_r( x)  \er +  \overline{u}_z( x)  \ez .
\end{aligned}
\end{equation}

Among the cylindrical components, only $\overline {u}_{  \theta}$ evolves with time. The stationary $rz$-components $\overline {u}_{  r}$ and $\overline {u}_{  z}$   are defined as follows. 
\begin{equation}\label{eq:def_overline_u_r_u_z}
\begin{aligned}
\overline {u}_{  r} (   z ,   r  ) & := - \ep^2 \mu^{ \frac{2}{p}  -s } \nu^{ \frac{1}{p}   } f '  (  \mu  \rho  )  \p_z \rho    \\
\overline {u}_{ z} (   z ,   r  ) &:=  \overline {u}_{z, p }  +  \overline {u}_{ z,c } 
\end{aligned}
\end{equation} 
where $\overline {u}_{ z, p }  $ is the principal part of $\ez$-component
\begin{equation}\label{eq:def_overline_u_zp}
\overline {u}_{ z, p }  : = \ep^2 \mu^{\frac{2}{p}  -s } \nu^{ \frac{1}{p}   }   f  '  (  \mu \rho  )   \p_r \rho 
\end{equation}
and $ \overline {u}_{ z, c } $ is a divergence corrector defined by 
\begin{equation}\label{eq:def_overline_u_zc}
\overline {u}_{ z, c }   : = \ep^2 \mu^{-1 +\frac{2}{p}  -s } \nu^{ \frac{1}{p}   }  \frac{f       (  \mu \rho  ) }{r}   .  
\end{equation}

Since  $\overline u  $ is an axisymmetric vector field, we check that it indeed has zero divergence:
\begin{equation}\label{eq:overline_u_zero_div}
\begin{aligned} 
\D \overline u   
&=  \p_r \overline {u}_{ r }   + \frac{ \overline {u}_{ r } }{r} + \p_z \overline {u}_{ z, p }  +   \p_z \overline {u}_{ z, c }     \\ 
& = 0 .
\end{aligned}
\end{equation}

The key point is that the vector field $ \RR^2 \ni ( \overline {u}_{  r} , \overline {u}_{  z, p}        ) = C_{\ep,\mu} \nabla^{\perp} (f (\mu \rho) )  $ is a radial vortex  and hence a stationary solution to the 2D Euler equation on  the $rz$-plane with compact support. More precisely, 

\begin{lemma}\label{lemma:overline_u_2D_Euler}
There exists a  pressure $\overline p (r,z)$, smooth in $(r,z) \in   \RR^+\times \RR$ and constant outside the support of $( \overline {u}_{  r} , \overline {u}_{  z,p}        )$, such that
\begin{equation}\label{eq:stationary_2Deuler}
\begin{cases}   
\big(  \overline {u}_{  r}  \p_r    +  \overline {u}_{ z, p}  \p_z  \big) \overline {u}_{  r}  + \p_r \overline p  = 0 & \\
\big(  \overline {u}_{  r}  \p_r    +  \overline {u}_{ z, p}  \p_z  \big)  \overline {u}_{ z, p}  +  \p_z  \overline p     = 0&\\
\p_r  \overline {u}_{  r}  +  \p_z \overline {u}_{ z, p} = 0 &
\end{cases}
\quad\text{in $(r,z) \in  \RR^+\times \RR$} .
\end{equation}

In addition, for any smooth $G(r,z)$ there holds
\begin{equation}\label{eq:stationary_2Deuler_b}
\big(  \overline {u}_{  r}  \p_r    +  \overline {u}_{ z, p}  \p_z  \big) G = \ep^2 \mu^{ -s    +\frac{2}{p } } \nu^{ \frac{1}{p}}  \frac{f'(\rho)}{\rho}   \p_\varphi G .
\end{equation}
\end{lemma}
\begin{proof}
Since $( \overline {u}_{  r} , \overline {u}_{ z, p}        )  = C_{\ep ,\mu}  \nabla_{r,z }^\perp \left(f  (\mu \rho ) \right)$, the vorticity $\overline {w}  = \nabla_{r,z }^\perp \cdot ( \overline {u}_{  r} , \overline {u}_{ z, p}        )  = C_{\ep ,\mu} \Delta_{r,z} \left(f  (\mu \rho ) \right) $ and hence the vorticity is compactly supported and radial in $\rho$.

The second \eqref{eq:stationary_2Deuler_b} follows from $-\p_z \rho \p_r + \p_r \rho \p_z= \rho^{-1} \p_\varphi $.
\end{proof}

\subsection{Definition of \texorpdfstring{$\overline{u}_\theta$}{utheta}}
Motivated by \eqref{eq:stationary_2Deuler_b}, we give the construction of the $\theta$ component, the one that manifests the desired norm inflation.

Define  the $\theta$-component  $ \overline{u}_\theta:\RR^+ \times \RR^3 \to \RR  $ as the unique smooth solution of the free transport equation on $( r,z) \in \RR^+ \times \RR  $
\begin{equation}\label{eq:def_overline_u_theta}
\begin{cases}
\p_t \overline{u}_\theta +   (  \overline {u}_{  r} \p_r  +   \overline {u}_{ z, p} \p_z   )   \overline{u}_\theta = 0 &\\
\overline{u}_\theta  |_{t = 0} =   {u}_{0,\theta} 
\end{cases}
\text{for $(t, r,z) \in \RR^+ \times \RR^+ \times \RR  $}
\end{equation}
where the initial data
\begin{equation}\label{eq:def_overline_u_theta0}
{u}_{0,\theta} := 
\begin{cases}
\ep^{-2 + k_0 N } \mu^{\frac{2}{p } -  s - k_0  } \nu^{ \frac{1}{p}}  \p_\rho^{k_0 }   \left[   \sin(\varphi +  \zeta(\rho) )  g(\mu \rho )\right]  &\quad \text{if $s<0$} \\
\ep^2 \mu^{\frac{2}{p }- s    } \nu^{ \frac{1}{p}}    \sin(\varphi    )  g(\mu \rho ) & \quad \text{if $s > 0$} .
\end{cases}
\end{equation}
where $\p_\rho^{k_0 }  $ is $k_0$ times differentiation in $\rho$  and $  \zeta(\rho): = t^*   \ep^2 \mu^{ -s    +\frac{2}{p } } \nu^{ \frac{1}{p}} \rho^{-1} = \ep^{-N} \mu^{-1} \rho^{-1}$.

By   the design of $ {u}_{0,\theta}  $ and Lemma \ref{lemma:overline_u_2D_Euler},  we can express $ \overline{u}_\theta $ explicitly 
\begin{equation}\label{eq:def_overline_u_theta_1}
\overline{u}_\theta (t ) =   
\ep^2 \mu^{\frac{2}{p }- s    } \nu^{ \frac{1}{p}}    \sin(\varphi  - t   \ep^2 \mu^{ -s  +\frac{2}{p } } \nu^{ \frac{1}{p}} \rho^{-1}   )  g(\mu \rho )     \quad \text{if $s > 0$} .
\end{equation} 

A similar but more complicated expression holds for $s<0$ where one see that the phase angle  $\varphi +  \zeta(\rho)  $ gets unmixed at $t =t^*$.

\subsection{Properties of the approximate solution}

We start with immediate consequences of the definitions \eqref{eq:def_overline_u}--\eqref{eq:def_overline_u_zc} and derive useful Sobolev estimates for the approximate solution.

By their definitions, the objects $\overline{u}_{ \theta } , \overline{u}_{  r } , \overline{u}_{  z }   $ are supported in the ring $\{ x \in \RR^3: \rho \leq 2 \mu^{-1}\}$.  Therefore, the vector field $\overline{u} : \RR^+ \times \RR^3  \to \RR^3$ defined by  \eqref{eq:def_overline_u}--\eqref{eq:def_overline_u_zc} is smooth, divergence-free, and compactly supported  with a support of size $\sim \mu^{-2} \nu^{-1}$.

We derive useful Sobolev estimates for $\overline{u} $ of positive order in the below lemma. All Sobolev or Lebesgue norms in this paper are taken in $\RR^3$ and hence we will not spell out the spatial domain.

\begin{lemma}\label{lemma:estimates_overline_u}
For any $k \in \NN$ and $1\leq {q}\leq \infty$, $\overline{u} $ satisfies   the estimates
\begin{equation}\label{eq:lemma_u_0_1}
|\overline{u}  |_{ L^\infty( [0,t^*]; W^{k,{q}})  } \leq C_{ \ep, k,{q} }  \mu^{ k -s  }  \Big( \mu^{\frac{2}{p } - \frac{2}{{q} } } \nu^{\frac{1}{p   } - \frac{1}{{q} } }  \Big) ,
\end{equation}
where the constant $C_{\ep, k, {q}} > 0 $ is independent of $  \mu$ and $\nu$.

\end{lemma}
\begin{proof}
Since $\overline{u}$ is supported on a set of size $\sim\mu^{-2}\nu^{-1}$, it suffices to prove that for $   k \in \NN$,
\begin{equation}
| \nabla^{k} \overline{u}  |_{L^\infty_{t,x}([0,t^*] \times \RR^3)   }  \leq     C_{ \ep, k }   \mu^{ k -s  } \mu^{ \frac{ 2}{p}  } \nu^{\frac{1}{p}   }   .   
\end{equation}
and we analyze each component separately.

\noindent
\textbf{\underline{Part 1: Estimates of $\overline{u}_{  r } \er$ and $\overline{u}_{  z } \ez$}}

We drop the time variable as these two vector fields are stationary.

Since the constants in the estimates are allowed to depend on $\ep$, by the definitions \eqref{eq:def_overline_u_r_u_z}, \eqref{eq:def_overline_u_zc}, and  \eqref{eq:def_overline_u_zp} and using \eqref{eq:diff_rho_varphi} for differentiation  in $\rho\phi$-coordinates, we  obtain  
\begin{equation}\label{eq:est_u_components_r_zp}
\begin{aligned}
|  \nabla^{k}  \overline{u}_{r }  |_{L^{\infty } } & \leq  C_{\ep,k } \mu^{ k -s }    \Big( \mu^{\frac{2}{p }   } \nu^{\frac{1}{p}  }  \Big) \\
|  \nabla^{k} \overline{u}_{z,p } |_{L^{\infty } } & \leq  C_{\ep,k } \mu^{ k -s }    \Big( \mu^{\frac{2}{p }   } \nu^{\frac{1}{p}  }  \Big)  
\end{aligned}
\end{equation} 
and
\begin{equation}\label{eq:est_u_components_zc}
\begin{aligned}
|  \nabla^{k} \overline{u}_{ z,c }  |_{L^{\infty } } & \leq  C_{\ep,k }  (\mu^{-1} \nu)   \mu^{ k -s }   \Big( \mu^{\frac{2}{p }   } \nu^{\frac{1}{p}  }  \Big)  .
\end{aligned}
\end{equation} 

Then the estimates for   $\overline{u}_{  r } \er$ and $\overline{u}_{  z } \ez$   follow from \eqref{eq:est_u_components_r_zp}, \eqref{eq:est_u_components_zc}, and $|\nabla^{k} \er| \lesssim r^{-k} \lesssim \nu^{k}$ on the support of $\overline{u}$ thanks to $\nu \ll \mu$.

\noindent
\textbf{\underline{Part 2: Estimates of $\overline{u}_\theta \et$ }}

Since $ \overline{ u}_{ \theta}$ is   transported by $\overline{ u}_{ r}$  and $\overline{ u}_{ z,p} $, in both cases $s>0$ and $s<0$ by Lemma \ref{lemma:overline_u_2D_Euler} we can write synthetically the  $\overline{ u}_{ \theta} $ as
\begin{equation}\label{eq:aux_lemma_error_5}
\overline{ u}_{ \theta}  =   \mu^{\frac{2}{p }-  s   } \nu^{ \frac{1}{p}} \sum_{i}  a_i F_i(\mu \rho )  G_i (\varphi + b_{i}(t) \rho^{-1} \mu^{-1 } )   
\end{equation}
where $F_i  $ and $G_i$ are finitely many smooth profiles that can be computed explicitly and $a_i, b_i(t) \in \RR $ are $\ep$ dependent but uniformly bounded on $[0,t^*]$ in $\nu,\mu$. Theses claims follow  from the definition of $t^*$ in \eqref{eq:def_critical_t*}.

When $k  =0$, the estimates follow trivially. For integers $k\geq 1$, we need to bound the factor of each (spatial) differentiation on $\overline{ u}_{ \theta}$.

By the construction of $ \overline{ u}_{ \theta} $, all $F_i$ in \eqref{eq:aux_lemma_error_5} are supported in $\rho \in [1 , \frac{3}{2}]$, so by differentiating in $\rho\phi$ we observe from \eqref{eq:diff_rho_varphi} that the maximum factor for each differentiation is $ C_\ep   \mu $, namely
\begin{equation}\label{eq:aux_lemma_error_7}
| \nabla^{k}(  \overline{u}_{   \theta }  )|_{L^\infty([0,t^*] \times \RR^3 )} \leq   C_{  \ep, k }   \mu^{ k -s  }  \mu^{\frac{2}{p}   } \nu^{ \frac{1}{p}   } .
\end{equation}

It follows from  \eqref{eq:aux_lemma_error_7} and $ |\nabla^k \et| \lesssim \nu^{k}$ on the support of $ \overline{u}_{   \theta } $ that
\begin{equation}\label{eq:aux_lemma_error_8}
| \nabla^{k}(  \overline{u}_{   \theta }  \et )|_{L^\infty([0,t^*] \times \RR^3 )} \leq   C_{  \ep, k }   \mu^{ k -s  }  \mu^{\frac{2}{p}   } \nu^{ \frac{1}{p}   } .
\end{equation}

\end{proof}

\subsection{Proof of Proposition \ref{prop:approximate}}
\begin{proof}[Proof of Proposition \ref{prop:approximate}] 

Since the estimate \eqref{eq:prop_error_1} have been shown in Lemma \ref{lemma:estimates_overline_u}, it only remains to show \eqref{eq:prop_error_2} and \eqref{eq:prop_error_3}, the construction of $ \overline{E}$ and its estimates.

\noindent
\textbf{\underline{Step 1: Definition of $\overline{p}$ and $\overline{E}$}}

The construction of the error field $\overline{E} $   splits into two parts
\begin{equation}\label{eq:aux_proof_error_1}
\overline{E} = \overline{E}_{e} +  \overline{E}_{v}
\end{equation}
where $\overline{E}_{e}  $ is the ``Eulerian'' error  defined via cylindrical components  $\overline{E}_{e} = \overline{E}_{e,\theta } \et  +  \overline{E}_{e, r } \er  + \overline{E}_{e, z }  \ez $ by  
\begin{align}\label{eq:aux_proof_error_2}
\overline{E}_{e,\theta }  & = \overline{u}_{z, c} \p_z \overline{u}_{  \theta } +  \frac{1}{r} \overline{u}_{  \theta } \overline{u}_{   r}   \\
\overline{E}_{e,r } & =  \overline{u}_{z,c} \p_z  \overline{u}_{   r }- \frac{1}{r} \overline{u}_{  \theta }^2  \\
\overline{E}_{e,z } &= \overline{u}_{  r} \p_r \overline{u}_{z,c} + \overline{u}_{   z,p } \p_z  \overline{u}_{z, c}+     \overline{u}_{z, c} \p_z  \overline{u}_{z }
\end{align}
and $\overline{E}_{v}$ is the ``viscous'' error given by
\begin{equation}\label{eq:aux_proof_error_2b}
\overline{E}_{v} = -\Delta \overline{u} .
\end{equation}

We define the pressure $ \overline{p}(r,z)  $   as \eqref{eq:stationary_2Deuler}, which  is also smooth as a function $\RR^3 \to \RR$.

Note that $\overline{E}_{e}  $ corresponds to lower order nonlinear errors arising between the full 3D Euler dynamics and the transport plus  2D Euler dynamics of the approximate solution, whereas  $\overline{E}_{v}$ is introduced when neglecting the viscous dissipation.

By  Lemma \ref{lemma:overline_u_2D_Euler} and \eqref{eq:def_overline_u_theta}, $\overline{ u}$ satisfies the   3D Euler equations with $\overline{E}_{e }  $ as the error on the right:
\begin{equation}\label{eq:aux_proof_error_3}
\p_t  \overline{u}  + \overline{u}\cdot \nabla \overline{u}   + \nabla \overline{p} = \overline{E}_e
\end{equation}
which can be seen in cylindrical coordinates, namely
\begin{equation}\label{eq:aux_proof_error_4}
\begin{cases}
\p_t \overline{u}_{  \theta } + \overline{u}_{  r} \p_r \overline{u}_{  \theta } + \overline{u}_{  z} \p_z \overline{u}_{  \theta }+ \frac{1}{r} \overline{u}_{  \theta }\overline{u}_{  r}  = \overline{E}_{e,\theta } & \\
\p_t \overline{u}_{  r } + \overline{u}_{  r} \p_r \overline{u}_{  r } + \overline{u}_{  z} \p_z \overline{u}_{  r } - \frac{1}{r} \overline{u}_{  \theta }^2 + \p_r \overline{p} = \overline{E}_{e, r  } & \\
\p_t \overline{u}_{  z} + \overline{u}_{  r} \p_r \overline{u}_{ z} + \overline{u}_{  z} \p_z \overline{u}_{  z}  +  \p_z   \overline{p}  = \overline{E}_{e, z } .&  
\end{cases}
\end{equation}

It follows from  \eqref{eq:aux_proof_error_3} that $\overline{ u}$ satisfies the 3D Navier-Stokes equations with $\overline{E} = \overline{E}_{e }  + \overline{E}_{v }$ as the error, which proves \eqref{eq:prop_error_2}.

\noindent
\textbf{\underline{Step 2: Estimates of  $\overline{E}$}}

It remains to show the estimate for $\overline{E}   $. By the considerations in Lemma \ref{lemma:estimates_overline_u}, it suffices to only consider the case $k =0$ for $ | \overline{E} |_{L^\infty } $, that is
\begin{equation}\label{eq:aux_proof_error_5}
| \overline{E} |_{L^\infty([0,t^*] \times \RR^3 ) }  \lesssim_\ep \mu^{-1}\nu  \big(\mu^{ 1-s+ \frac{ 2}{p}} \nu^{\frac{ 1}{p} } \big)  \big( \mu^{-s + \frac{ 2}{p} } \nu^{\frac{ 1}{p} } \big). 
\end{equation}

For the viscous error $\overline{E}_v $ of \eqref{eq:aux_proof_error_2b}, we apply Lemma \ref{lemma:estimates_overline_u}  
\begin{equation}\label{eq:aux_proof_error_6}
| \overline{E}_v |_{L^\infty([0,t^*] \times \RR^3 ) } \leq   | \nabla^2 \overline{u}  |_{L^\infty ([0,t^*] \times \RR^3 ) } \leq C_{ \ep  }  \mu^{ 2  -s  }  \Big( \mu^{\frac{2}{p }   } \nu^{\frac{1}{p   }   }  \Big) .
\end{equation}
Here we notice that
\begin{equation}\label{eq:aux_proof_error_6aab}
\mu^{ 2  -s  }  \Big( \mu^{\frac{2}{p }   } \nu^{\frac{1}{p   }   }  \Big)  = \text{right hand side of \eqref{eq:aux_proof_error_5}} \times  \mu^{1+ s - \frac{ 3}{p} }\big( \mu^{-1} \nu \big)^{-1 - \frac{ 1}{p} } .  
\end{equation}
By \eqref{eq:def_small_b}, the last factor in \eqref{eq:aux_proof_error_6aab} satisfies
$$
\mu^{1+ s - \frac{ 3}{p} }\big( \mu^{-1} \nu \big)^{-1 - \frac{ 1}{p} } \leq  \mu^{-10 b} \mu^{ 2 b} \leq \mu^{-1}\nu ,
$$
and hence from \eqref{eq:aux_proof_error_6} and \eqref{eq:aux_proof_error_6aab} we conclude that
\begin{equation}\label{eq:aux_proof_error_6a}
| \overline{E}_v |_{L^\infty([0,t^*] \times \RR^3 ) } \leq C_\ep \mu^{-1}\nu  \big(\mu^{ 1-s+ \frac{ 2}{p}} \nu^{\frac{ 1}{p} } \big)  \big( \mu^{-s + \frac{ 2}{p} } \nu^{\frac{ 1}{p} } \big). 
\end{equation}

For the Eulerian error $\overline{E}_e $ given by \eqref{eq:aux_proof_error_2}, since   estimating  $\overline{E}_{e,\theta}$ and $\overline{E}_{e,z}$ is very  similar, we only demonstrate the estimate of $\overline{E}_{e,r} $.

For $ \overline{E}_{e,r}$, we   use the H\"older inequality and   the fact that $r \sim \nu^{-1}$ on the support of $\overline{u}  $, 
\begin{equation}\label{eq:aux_proof_error_7}
\begin{aligned}
| \overline{E}_{e,r} |_{L^\infty([0,t^*] \times \RR^3 ) }&  \leq |  \overline{u}_{z,c}|_{L^\infty } | \p_z  \overline{u}_{  r }  |_{L^\infty  } +    | \frac{1}{r}  \overline{u}_{    \theta }   |_{L^\infty }   |\overline{u}_{   \theta }   |_{L^\infty  }\\
&  \lesssim |  \overline{u}_{z, c}|_{L^\infty } |\nabla  \overline{u}_{  r }  |_{L^\infty  }+ \nu   |   \overline{u}_{    \theta }   |_{L^\infty }   |\overline{u}_{   \theta }   |_{L^\infty  }.
\end{aligned}
\end{equation}
By   \eqref{eq:est_u_components_zc}, and  \eqref{eq:prop_error_1} proved in Part 1,
\begin{equation}\label{eq:aux_proof_error_8}
\begin{aligned}
| \overline{E}_{e,r} |_{L^\infty([0,t^*] \times \RR^3 ) }&      \lesssim_\ep \big( \mu^{-1}\nu  \mu^{ -s} \mu^{\frac{ 2}{p} } \nu^{\frac{ 1}{p} } \big)  \big( \mu^{1-s} \mu^\frac{ 2}{p}  \nu^{ \frac{1 }{p}}  \big) + \nu \big( \mu^{-s} \mu^\frac{ 2}{p}  \nu^{ \frac{1 }{p}}  \big)^2   \\
&\lesssim_\ep   \mu^{-1}\nu  \big(\mu^{ 1-s+ \frac{ 2}{p}} \nu^{\frac{ 1}{p} } \big)  \big( \mu^{-s + \frac{ 2}{p} } \nu^{\frac{ 1}{p} } \big).  
\end{aligned}
\end{equation}
Using the same strategy, one can show that both $\overline{E}_{e,\theta}$ and $\overline{E}_{e,z}$ satisfy the same estimate as \eqref{eq:aux_proof_error_8}.

Combining \eqref{eq:aux_proof_error_6a} and \eqref{eq:aux_proof_error_8} gives \eqref{eq:aux_proof_error_5}.

\end{proof}

\section{Growth of the approximate solution: \texorpdfstring{$s>0$}{s>0}}\label{sec:overline_u_inflation_s>0}
In this and the next sections, we show the approximate solution $\overline{ u}$ constructed in the previous section develop the desired norm inflation on $[0,t^*]$.

Due to   differences in the growth mechanisms, we separate the two cases $s>0$ and $s<0$---this section focuses on the case $s>0$  whereas $s<0$ is more involved and will be treated in the next section.

The main result of this section is summarized in the following.

\begin{proposition}[Norm inflation for $\overline{u}$: $s>0$]\label{prop:approximate_inflation_s>0}
Let $s,p $ be given as in Theorem \ref{thm:Besov} such that $s> 0$. The approximate solution $\overline{u}$ constructed in Proposition \ref{prop:approximate} satisfies:  
\begin{enumerate}
\item  \textbf{Small initial data:} The initial data $\overline{u}_0  $ is small in $ \dot B^{s}_{p,1} (\RR^3)$:
\begin{equation}\label{eq:prop:approximate_inflation_s>0_1}
|  \overline{u}_0    |_{  \dot B^{s}_{p,1}} \lesssim  \ep^{   2 }.
\end{equation}

\item \textbf{Norm inflation at critical time:} There exists $\ep_0 = \ep_0(s,p )>0$ such that for all $ 0< \ep \leq \ep_0$ at the critical time $t^*= t^*(\ep,\mu)>0$ (defined in \eqref{eq:def_critical_t*}), $\overline{u}$ exhibits $\dot B^{s}_{p, \infty }(\RR^3)$   inflation:
\begin{equation}\label{eq:prop:approximate_inflation_s>0_2}
|  \overline{u} (t^*)   |_{\dot B^{s}_{p,\infty } } \gtrsim  \ep^{ - 2 } .
\end{equation}

\end{enumerate}
Both implicit   constants are independent of $\ep$ or $\mu,\nu$.
\end{proposition}

To prove Proposition \ref{prop:approximate_inflation_s>0} and for later use, we need the following elementary result. 

\begin{lemma}\label{lemma:G_estimates}
Let $h: \TT \to \RR$ be smooth and $2\pi$-periodic. Denote by $ \zeta  (\rho )    =  \ep^{ -N}  \mu^{  -1 }   \rho^{-1}   $. Let $G: \RR^3 \to \RR$ be defined in the toroidal $\rho \varphi$ coordinates by $ G = h (\varphi - \zeta  (\rho ) ) g(\mu \rho )$ with $g$ from \eqref{eq:def_profiles}.

Then for any $k \in \NN$ and $1 \leq q \leq \infty$, there exists $\ep_k > 0 $ (depending on $h$ and $g$) such that for any $0 < \ep \leq \ep_k $
\begin{equation}
| \p_\rho^k G |_{L^{ q}(\RR^3 )} \sim  \ep^{-N k }\mu^{ k - \frac{2}{q}} \nu^{ - \frac{1}{q}}
\end{equation}
where the implicit  constant does not depend on $\ep,\mu$ or $\nu$.

\end{lemma}
\begin{proof}
Since $g$ is supported  on $  \rho \subset [1, \frac{3}{2}] $,  $G$ is smooth and compactly away from $\rho =0 $ on $ \RR^3$. We need to estimate
\begin{equation} \label{eq:axu_lemma:G_estimates_0}
\begin{aligned}
\p_\rho^{ k}  G & =      \sum_{0\leq i \leq k  }   \mu^{ k-i} g^{(k-i)}(\mu \rho ) \p_\rho^{ i}   \left( h (\varphi - \zeta  (\rho ) ) \right).
\end{aligned}
\end{equation}

Let us consider the main term of \eqref{eq:axu_lemma:G_estimates_0}
\begin{equation}\label{eq:axu_lemma:G_estimates_2}
G_{k, p } : =     h^{(k )}(\varphi - \zeta  (\rho )  )  \left( \zeta^{(1)}  \right)^{k  }   g (\mu \rho ),
\end{equation}
and the decomposition 
\begin{equation}\label{eq:axu_lemma:G_estimates_3}
\p_\rho^{ k}  G     = G_{k, p }  + G_{k, c } .
\end{equation}

By Faà di Bruno's formula, for any $i\in \NN$, we see that
\begin{equation}\label{eq:axu_lemma:G_estimates_1}
\p_\rho^{ i}   \left( h (\varphi - \zeta  (\rho ) ) \right)    = \sum_{m_1,\dots,m_i} C_{  m_1,\dots,m_i} h^{(m_1+\dots+m_i)}(\varphi - \zeta  (\rho )  ) \Pi_{ 1\leq j \leq  i }  \left( \zeta^{(j)}  \right)^{m_j},
\end{equation}
where the summation runs over $m_1 + 2m_2 + \dots + i m_i = i $.

Observe that each $\zeta^{(j)}(\rho) $ is of the form $ C_j  \ep^{ -N}  \mu^{-1}   \rho^{-1 -j}  $. So    on the support of $g(\mu\rho)$, we have $ |\zeta^{(j)}(\rho)| \sim  \ep^{ -N}  \mu^{ j}   $.

It follows from \eqref{eq:axu_lemma:G_estimates_1} that in $ G_{k, c }  $, the  largest possible factor of $\ep$ is $(\ep^{ -N}   )^{k  - 1 }  $, namely
\begin{equation}\label{eq:axu_lemma:G_estimates_4}
\begin{aligned}
| G_{k, c }   |_{L^{ \infty }} &  \lesssim   (\ep^{ -N}   )^{k - 1 }      \mu^{k }   .
\end{aligned}
\end{equation}
Using the support property of $g(\mu \rho)$ we obtain from \eqref{eq:axu_lemma:G_estimates_4} that
\begin{equation}\label{eq:axu_lemma:G_estimates_44}
\begin{aligned}
| G_{k, c }   |_{L^{ q }} &  \lesssim   (\ep^{ -N}   )^{k - 1 }       \mu^{k -\frac{2}{q} }  \nu^{  -\frac{1}{q} }  .     
\end{aligned}
\end{equation}

On the other hand, the main term \eqref{eq:axu_lemma:G_estimates_2} satisfies the   bound
\begin{equation} \label{eq:axu_lemma:G_estimates_5}
\begin{aligned}
| G_{k, p } |_{L^{ q }} 
& \sim (\ep^{ -N}   )^{k    }     \mu^{k -\frac{2}{q} }  \nu^{  -\frac{1}{q} }       .
\end{aligned}
\end{equation}
Since the constants in \eqref{eq:axu_lemma:G_estimates_44} and \eqref{eq:axu_lemma:G_estimates_5} are independent of $\ep$, for each $k \in \NN$ we can choose $\ep_k > 0 $ such that  $| G_{k, c }   |_{L^{ q}}  \ll  | G_{k, p } |_{L^{ q}}  $, which implies 
\begin{equation}
| \p_\rho^{ k}  G |_{L^{ q }}     \sim  | G_{k, p }  |_{L^{ p}} \sim (\ep^{ -N}   )^{k    }     \mu^{k -\frac{2}{q} }  \nu^{  -\frac{1}{q} }       .
\end{equation}

\end{proof}

With the help of Lemma \ref{lemma:G_estimates}, we can finish the 

\begin{proof}[Proof of Proposition \ref{prop:approximate_inflation_s>0}]
Since $\overline{u}_0 =  \overline{u}_{0,\theta} \et + \overline{u}_r \er + \overline{u}_z \ez   $, by repeating the argument of Lemma \ref{lemma:estimates_overline_u} while keeping track of $\ep$ we can obtain
\begin{equation}\label{eq:aux_prop:approximate_inflation_s>0_1}
\begin{aligned}
|   \overline{u}_r  \er   |_{\dot  W^{k, p }    }  & \lesssim  \ep^{   2 } \mu^{k-s }\\
|    \overline{u}_z    \ez  |_{  \dot  W^{k, p }    }   & \lesssim  \ep^{   2 } \mu^{k-s }\\
|    \overline{u}_{0,\theta}   \et   |_{\dot  W^{k, p }     }   & \lesssim  \ep^{   2 }\mu^{k-s }
\end{aligned}
\end{equation}
for all $k \in \NN $.  So the upper bound \eqref{eq:prop:approximate_inflation_s>0_1} follows by applying Lemma \ref{lemma:besov_interpolation} with $s_1 =0 <s$ and $s_2 =k_0 >s$:
\begin{equation}\label{eq:aux_prop:approximate_inflation_s>0_1ab}
|  \overline{u}_0    |_{  \dot B^{s}_{p,1}} \lesssim   |  \overline{u}_0    |_{  \dot B^{0}_{p,\infty}}^{1- \frac{s}{k_0}} |  \overline{u}_0    |_{  \dot B^{k_0}_{p,\infty }}^{1- \frac{s}{k_0}}  \lesssim   |  \overline{u}_0    |_{  \dot L^{p}}^{1- \frac{s}{k_0}} |  \overline{u}_0    |_{  \dot W^{k_0, p  }}^{1- \frac{s}{k_0}} \lesssim \ep^2 .
\end{equation}

For the upper bound \eqref{eq:prop:approximate_inflation_s>0_2}, since only $ \overline{u}_\theta$ evolves with time, by \eqref{eq:aux_prop:approximate_inflation_s>0_1ab} it suffices to prove
\begin{equation} \label{eq:aux_prop:approximate_inflation_s>0_2}
|  \overline{u}_\theta  (t^*)  \et |_{\dot B^{s}_{p,\infty } } \gtrsim  \ep^{ - 2 } .
\end{equation}
In what follows we will first show that for any $k\in \NN$ if $\ep>0$ is sufficiently small, then
\begin{align}\label{eq:aux_prop:approximate_inflation_s>0_3b}
|  \overline{u}_\theta  (t^*)  \et |_{\dot W^{k , p }  } & \sim \ep^{2- k N } \mu^{k    -s }  
\end{align}
and then pass to  \eqref{eq:aux_prop:approximate_inflation_s>0_2} using Lemma \ref{lemma:besov_interpolation}.

We now focus on showing   \eqref{eq:aux_prop:approximate_inflation_s>0_3b}.
Thanks to the fact $ | \nabla^{k} \et |\lesssim \nu^{k} $ on the support of $\overline{u}_\theta  (t^*)  $, we see that  the estimate of $|  \nabla^{k }(\overline{u}_\theta  (t^*)  \et )  |_{L^{p } }$ is   dominated by $|  \nabla^{k_0} \overline{u}_\theta  (t^*)     |_{L^{p } }$.

Based on these considerations and \eqref{eq:diff_rho_varphi}, we apply Lemma \ref{lemma:G_estimates} to $\p_\varphi^j \overline{u}_{ \theta} (t^* )   $ and obtain that for any $i \in \NN$, there exists $\ep_i > 0 $ such that for any $0 < \ep \leq \ep_i $
\begin{equation}\label{eq:aux_prop:approximate_inflation_s>0_5}
\Big| \p_\rho^{ i} \p_\varphi^j ( \overline{u}_{ \theta} (t^* )   )  \Big|_{L^p }  \sim  \ep^{2-Ni} \mu^{ i  -s }.
\end{equation}

Now thanks to \eqref{eq:aux_prop:approximate_inflation_s>0_5}, the   lower bound of \eqref{eq:aux_prop:approximate_inflation_s>0_3b} follows
\begin{align}
\Big|\nabla^k  \big( \overline{u}_\theta  (t^*)  \et \big) \Big|_{L^{ p }  } & \geq  \Big|\p_\rho^k  \big( \overline{u}_\theta  (t^*)  \et \big) \Big|_{L^{ p }  } \\
& \geq  |    \p_\rho^k   \overline{u}_\theta  (t^*)    |_{  L^{p}  }   \\
& \gtrsim  \ep^{2- k N } \mu^{ 1 -s }.
\end{align}

Similarly the   upper bound of \eqref{eq:aux_prop:approximate_inflation_s>0_3b} also follows from \eqref{eq:aux_prop:approximate_inflation_s>0_5}:
\begin{align}
|\nabla^{k }  ( \overline{u}_\theta  (t^*)  \et) |_{  L^{p } } & \lesssim   \sum_{ 0\leq i \leq k }| \nabla^{i }    \overline{u}_\theta  (t^*)  \nabla^{k  -i  }\et |_{  L^{p}  }     \nonumber \\
& \lesssim \sum_{ 0\leq i\leq k } \nu^{k - i }| \nabla^{i }       \overline{u}_\theta  (t^*)    |_{  L^{p}  }    \nonumber \\
& \lesssim  \sum_{ 0\leq i \leq k }  \nu^{k  - i } \ep^{2- i N } \mu^{ i -s } \nonumber\\
& \lesssim   \ep^{2- k  N } \mu^{ k  -s } \label{eq:aux_prop:approximate_inflation_s>0_7},
\end{align}  
where we used $|\nabla^{k  - i  }\et| \lesssim \nu^{k -  i } $ on the support of $\overline{u}_\theta $   in the second inequality and \eqref{eq:aux_prop:approximate_inflation_s>0_5} with \eqref{eq:diff_rho_varphi}  in the third inequality.

Thus both directions of \eqref{eq:aux_prop:approximate_inflation_s>0_3b} have been established. Since  $   s< k_0 < 2k_0$, by Lemma \ref{lemma:besov_interpolation}
\begin{align}\label{eq:aux_prop:approximate_inflation_s>0_8}
|\overline{u} (t^*)   |_{\dot W^{k_0 , p  } }   & \lesssim   | \overline{u} (t^*)   |_{\dot B^{s}_{p,\infty } }^\frac{k_0  }{ 2k_0 -s} | \overline{u} (t^*)   |_{\dot W^{ 2 k_0 , p  } }^{  \frac{k_0 -s }{2k_0 -s }}  ,
\end{align}
and therefore similar to \eqref{eq:aux_prop:approximate_inflation_s>0_1ab},  by \eqref{eq:aux_prop:approximate_inflation_s>0_7} and \eqref{eq:aux_prop:approximate_inflation_s>0_8} we can conclude that    for all sufficiently small $\ep>0$ 
\begin{align}\label{eq:aux_prop:approximate_inflation_s>0_9}
| \overline{u} (t^*)   |_{\dot B^{s}_{p,\infty } } & \gtrsim  | \overline{u} (t^*)   |_{\dot W^{ 2 k_0 , p  } }^{- \frac{k_0 -s }{k_0 }}  |\overline{u} (t^*)   |_{\dot W^{k_0 , p  } }^{  \frac{2k_0 -s }{k_0 }}  \\
& \gtrsim \ep^{2 - sN } \\
& \gtrsim \ep^{ - 2 }  
\end{align}
where we have used $s N  \geq 100$.

\end{proof}

\subsection{Estimates of vorticity}
We record some useful estimates for the vorticity of the approximate solution.
\begin{lemma}\label{lemma:vorticity_s>0}
Under the setting of Proposition \ref{prop:approximate_inflation_s>0}, for any $1 \leq q \leq \infty$, there holds,
\begin{equation}
| \overline{\omega } (t^*)|_{L^q } \gtrsim \ep^{-2} \mu^{1 -s}  \mu^{\frac{2}{p} -\frac{2}{q} }  \nu^{ \frac{1}{p} -\frac{1}{q} }      ,
\end{equation}
where $ \overline{\omega }: = \nabla \times \overline{u}  $ is the vorticity of $\overline{u} $.

\end{lemma}
\begin{proof}
Denote $\overline{\omega } = \overline{\omega }_\theta \et  + \overline{\omega }_r \er +    \overline{\omega }_z \ez $. It suffices to show the lower bounds for $ \overline{\omega }_r(t^*) = 
\p_z \overline{ u}_\theta (t^*)$.

Since $\p_z \overline{ u}_\theta = (\sin(\varphi) \p_\rho + \cos(\varphi) \frac{1}{\rho} \p_\varphi )\overline{ u}_\theta$, we obtain from \eqref{eq:aux_prop:approximate_inflation_s>0_5} that
\begin{align}
|\p_z \overline{ u}_\theta (t^*) |_{L^q} \geq |\p_\rho \overline{ u}_\theta (t^*)|_{L^q} - |\frac{1}{\rho} \p_\varphi \overline{ u}_\theta (t^*) |_{L^q}   \geq \big(C \ep^{2-N}  - c \ep^{2 }   \big) \mu^{1 -s}  \mu^{\frac{2}{p} -\frac{2}{q} }  \nu^{ \frac{1}{p} -\frac{1}{q} }.
\end{align}
Since $N \geq 100$, the first term above dominates the second, and hence we have
\begin{equation}
| \overline{\omega } (t^*)|_{L^q } \geq |\p_z \overline{ u}_\theta (t^*) |_{L^q} \gtrsim \ep^{-2} \mu^{1 -s}  \mu^{\frac{2}{p} -\frac{2}{q} }  \nu^{ \frac{1}{p} -\frac{1}{q} }      .
\end{equation}

\end{proof}

\section{Growth of the approximate solution: \texorpdfstring{$s<0$}{s<0}}\label{sec:overline_u_inflation_s<0}

In this section, we show the norm inflation for the approximate solution $ \overline{u} $ when $s<0 $. The estimates are more complex than those for $s>0$, as here we rely on duality to estimate norms with a negative derivative.

The main result of this section is the extension of Proposition \ref{prop:approximate_inflation_s>0} to the case $s<0$.

\begin{proposition}[Norm inflation for $\overline{u}$: $s<0$]\label{prop:approximate_inflation_s<0}
Let $s,p $ be given as in Theorem \ref{thm:Besov} such that $s<0$. The approximate solution $\overline{u}$ constructed by Proposition \ref{prop:approximate} satisfies  

\begin{enumerate}
\item  \textbf{Small initial data:} The initial data $\overline{u}_0  $ is small in $ \dot B^{s}_{p,1} (\RR^3)$:
\begin{equation}\label{eq:prop:approximate_inflation_1}
|  \overline{u}_0    |_{  \dot B^{s}_{p,1}} \lesssim  \ep^{   2 } .
\end{equation}

\item \textbf{Norm inflation at critical time:} There exists $\ep_0 = \ep_0(s,p )>0$ such that for all $ 0< \ep \leq \ep_0$ at the critical time $t^*= t^*(\ep,\mu)>0$ (defined in \eqref{eq:def_critical_t*}), $\overline{u}$ exhibits $\dot B^{s}_{p, \infty }(\RR^3)$   inflation:
\begin{equation}\label{eq:prop:approximate_inflation_2}
|  \overline{u} (t^*)   |_{\dot B^{s}_{p,\infty } } \gtrsim  \ep^{ - 2 }.
\end{equation}

\end{enumerate}
Both implicit   constants are independent of $\ep$ or $\mu,\nu$.
\end{proposition}

In the rest of this section, we prove Proposition \ref{prop:approximate_inflation_s<0} in several steps.

\subsection{Upper bounds for \texorpdfstring{$  \overline{u}_r  \er$}{ur} and \texorpdfstring{$ \overline{u}_z  \ez$}{uz}}

Recall that the initial data $ \overline{u}_0     =  \overline{u}_r  \er     +  \overline{u}_z  \ez    +  \overline{u}_{0,\theta }  \et    $. We start with the upper bound for  $\overline{u}_r  \er $. 

\begin{lemma}\label{lemma:est_u_0r}
Under the setting of Proposition \ref{prop:approximate_inflation_s<0}, there holds
\begin{align}
|  \overline{u}_r  \er   |_{ \dot B^{s}_{p,1}     }  & \lesssim  \ep^{   2 } .
\end{align}
\end{lemma}
\begin{proof}

Since $ - k_0 < s< 0 $, by Lemma \ref{lemma:besov_interpolation} it suffices to show  
\begin{equation}\label{eq:aux_lemma:est_u_0r_1}
| \overline{u}_r  \er |_{\dot{ B }^{   - k_0 }_{p,\infty }}  \lesssim \ep^2 \mu^{-  k_0  - s }  
\end{equation}
and
\begin{equation}\label{eq:aux_lemma:est_u_0r_2}
| \overline{u}_r  \er  |_{L^{p  }}  \lesssim \ep^2 \mu^{ - s }  .
\end{equation}

The $L^p$ bound \eqref{eq:aux_lemma:est_u_0r_2} follows directly from the definition of $ \overline{u}_r $, and we focus on the negative Besov bound \eqref{eq:aux_lemma:est_u_0r_1}.

To show  \eqref{eq:aux_lemma:est_u_0r_1}, we invoke the duality principle \eqref{eq:besov_duality}
\begin{equation}\label{eq:aux_lemma:est_u_0r_3}
| \overline{u}_r  \er |_{\dot{ B }^{   - k_0 }_{p, \infty  }} \sim  \sup_{  \phi \in Q^{k_0 }_{p',1}}  \Big| \int_{\RR^3} \overline{u}_r  \er \cdot \phi \, dx \Big| ,
\end{equation}
where
$Q^{k_0}_{p' ,1} : = \{  \phi  \in  \mathcal{S} (\RR^3): |\phi  |_{\dot{ B }^{     k_0 }_{p' ,1 }  }   \leq  1 \}$ with $ \mathcal{S}(\RR^3)$ the Schwartz class and $p' $ being the H\"older conjugate  of $p $.

For any given $  \phi  \in  \mathcal{S} (\RR^3)$, consider the integral on the right-hand side of \eqref{eq:aux_lemma:est_u_0r_3}. Integrating by parts in $z$, we have
\begin{equation}\label{eq:aux_lemma:est_u_0r_4}
\begin{aligned}
\int_{\RR^3} \overline{u}_r  \er    \cdot \phi \, dx & =  \ep^2 \mu^{-1+ \frac{2}{p}  -s } \nu^{ \frac{1}{p}   } \int_{\RR^3}        f    (  \mu  \rho  )    \er \cdot   \p_z \phi  \, dx \\
& : = \ep^2 \mu^{-1+ \frac{2}{p}  -s } \nu^{ \frac{1}{p}   }  I_\phi .
\end{aligned}
\end{equation}

To estimate the integral $I_\phi$, we will repeatedly integrate by parts in $  \rho  $.  To reduce notation, let us denote by $ f^{(-k)} \in C^\infty_c([\frac{1}{2} , \frac{3}{2}])$   the $k$-order anti-derivative of $f $ which exists up to $k=N \geq 100$ thanks to \eqref{eq:def_profiles}.

Using $\p_\rho \er =0$,   we  repeatedly integrate by parts $k_0 -1=5$ times in $\rho$ 
\begin{equation}\label{eq:aux_lemma:est_u_0r_5}
\begin{aligned} 
I_\phi &  = \int_{\RR^+ \times \TT \times\TT }        f    (  \mu  \rho  )    \er \cdot   \p_z \phi \, \, r \rho   d\rho d\varphi d \theta  \\
& = -\mu^{-k_0+1}\int_{\RR^+ \times \TT \times\TT }        f^{(-k_0+1)}   (  \mu  \rho  )    \er \cdot  \p_\rho^{k_0-1}  \big( r \rho  \p_z \phi   \big)  d\rho d\varphi d \theta .
\end{aligned}
\end{equation}
Since $\p_\rho r = \cos(\varphi)$, we have $\p_\rho^{k_0- 1 }  \big( \p_z \phi \, \, r \rho  \big) = r\rho \p_\rho^{k_0- 1 }   \p_z \phi + ( {k_0- 1 } ) ( \cos(\varphi) \rho + r )    \p_\rho^{ k_0- 2  }    \p_z \phi +  2C^2_{{k_0- 1 }} \cos(\varphi)       \p_\rho^{ k_0- 3  }  \p_z \phi    $.    

It follows from \eqref{eq:aux_lemma:est_u_0r_5} that $ I_\phi$ can be further decomposed into
\begin{equation}\label{eq:aux_lemma:est_u_0r_6}
\begin{aligned} 
I_\phi &  = - \mu^{-k_0+1} \big( I_1 + (k_0 - 1 ) I_2 + 2C^2_{{k_0- 1 }} I_3 \big)
\end{aligned}
\end{equation}
with 
\begin{align}
I_1 & = \int_{\RR^+ \times \TT \times\TT }        f^{(-k_0+1)}   (  \mu  \rho  )    \er \cdot     \p_\rho^{k_0-1}   \p_z \phi r\rho \,\,  d\rho d\varphi d \theta  \label{eq:aux_lemma:est_u_0r_7a} \\
I_2 & =   \int_{\RR^+ \times \TT \times\TT }        f^{(-k_0+1)}   (  \mu  \rho  )    \er \cdot      \p_\rho^{k_0-2}   \p_z \phi ( \cos(\varphi) \rho + r )  \,\,  d\rho d\varphi d \theta \label{eq:aux_lemma:est_u_0r_7b} \\
I_3 & =    \int_{\RR^+ \times \TT \times\TT }         f^{(-k_0+1)}   (  \mu  \rho  )    \er \cdot       \cos(\varphi)      \p_\rho^{k_0-3}   \p_z \phi  \,\,  d\rho d\varphi d \theta \label{eq:aux_lemma:est_u_0r_7c}.
\end{align}

For $I_1$, using  \eqref{eq:diff_rho_varphi_2} we obtain the point-wise bound 
\begin{equation}\label{eq:aux_lemma:est_u_0r_8}
\Big|  \p_\rho^{k_0-1}   \p_z \phi (x)\Big| \lesssim \Big|  \nabla^{k_0-1}   \p_z \phi (x)\Big| \leq \big|  \nabla^{k_0 }   \phi (x)\big| .
\end{equation}
Switching back to $\RR^3$,   and by \eqref{eq:aux_lemma:est_u_0r_8} we obtain that  
\begin{align} \label{eq:aux_lemma:est_u_0r_9}
|I_1|  \lesssim   \int_{\RR^3  }       \big|f^{(-k_0+1)}   (  \mu  \rho  )  \big|        \nabla^{k_0}\phi  \big|        dx.
\end{align}
Hence $I_1$  obeys the desired bound
\begin{align} \label{eq:aux_lemma:est_u_0r_10}
|I_1|  \lesssim  \mu^{  - \frac{2}{p}    } \nu^{-\frac{1}{p}}        \big|        \nabla^{k_0}\phi  \big|_{L^{{p'}} }   .
\end{align}

For $I_2$, we further integrate by parts in $\rho$ one more time to obtain
\begin{equation}\label{eq:aux_lemma:est_u_0r_11}
I_2   = -\mu^{-1} \big(I_{21} + 2 I_{22} \big)
\end{equation}
with 
\begin{align}
I_{21} & =      \int_{\RR^+ \times \TT \times\TT }        f^{(-k_0  )}   (  \mu  \rho  )    \er \cdot      \p_\rho^{k_0-1}   \p_z \phi ( \cos(\varphi) \rho + r )  \,\,  d\rho d\varphi d \theta \label{eq:aux_lemma:est_u_0r_12a} \\
I_{22}& =     \int_{\RR^+ \times \TT \times\TT }        f^{(-k_0  )}   (  \mu  \rho  )    \er \cdot      \p_\rho^{k_0-2}   \p_z \phi   \cos(\varphi)   \,\,  d\rho d\varphi d \theta \label{eq:aux_lemma:est_u_0r_12b}.
\end{align}

For $I_{21}$, we switch back to $\RR^3$ and use \eqref{eq:aux_lemma:est_u_0r_8} to obtain  
\begin{equation}\label{eq:aux_lemma:est_u_0r_13}
\begin{aligned} 
|I_{21} | & \lesssim  \int_{\RR^3 }     |   f^{(-k_0  )}   (  \mu  \rho  ) |   |   \nabla^{k_0} \phi|  \frac{\rho  + r }{\rho r}    \, dx \\
& \lesssim \mu^{ - \frac{2}{p}    } \nu^{ 1 -\frac{1}{p}}        \big|        \nabla^{k_0}\phi  \big|_{L^{{p'}} }   .
\end{aligned}
\end{equation}
For $I_{22}$, integrating by parts in $\rho$ once more
\begin{align}\label{eq:aux_lemma:est_u_0r_14}
I_{22}& =  -  \mu^{-  1 } \int_{\RR^+ \times \TT \times\TT }        f^{(-k_0 -1 )}   (  \mu  \rho  )    \er \cdot      \p_\rho^{k_0-1}   \p_z \phi   \cos(\varphi)   \,\,  d\rho d\varphi d \theta ,
\end{align}
and it follows from \eqref{eq:aux_lemma:est_u_0r_14} and \eqref{eq:aux_lemma:est_u_0r_8} that
\begin{equation}\label{eq:aux_lemma:est_u_0r_15}
\begin{aligned} 
|I_{22} | & \lesssim  \mu^{- 1  }\int_{\RR^3 }     |   f^{(-k_0 -1 )}   (  \mu  \rho  ) |   |   \nabla^{k_0} \phi|  \frac{1 }{\rho r}    \, dx  \\
& \lesssim    \mu^{  - \frac{2}{p}    } \nu^{ 1  -\frac{1}{p}}        \big|        \nabla^{k_0}\phi  \big|_{L^{{p'}} }   .
\end{aligned}
\end{equation}
Thus by \eqref{eq:aux_lemma:est_u_0r_11}, \eqref{eq:aux_lemma:est_u_0r_13}, and \eqref{eq:aux_lemma:est_u_0r_15} we have 
\begin{equation}\label{eq:aux_lemma:est_u_0r_16}
|I_2| \lesssim  \big( \mu^{- 1    } \nu \big)   \mu^{   - \frac{2}{p}    } \nu^{-\frac{1}{p}}        \big|        \nabla^{k_0}\phi  \big|_{L^{{p'}}(\RR^3)}   .
\end{equation}

Finally noting that $| I_3| \sim  \mu^{-1} | I_{22}| $, by \eqref{eq:aux_lemma:est_u_0r_4}, \eqref{eq:aux_lemma:est_u_0r_6}, \eqref{eq:aux_lemma:est_u_0r_10}, and \eqref{eq:aux_lemma:est_u_0r_16} we have  
\begin{equation}\label{eq:aux_lemma:est_u_0r_17} 
\begin{aligned} 
\Big| \int_{\RR^3} \overline{u}_r  \er    \cdot \phi \, dx\Big| & \lesssim  \ep^2 \mu^{-1+ \frac{2}{p}  -s } \nu^{ \frac{1}{p}   } | I_\phi| \\
& \lesssim  \ep^2 \mu^{-k_0 + \frac{2}{p}  -s } \nu^{ \frac{1}{p}   } (| I_1|+ | I_1|+ | I_3|) \\ 
& \lesssim \ep^2 \mu^{-k_0   }    \big|        \nabla^{k_0}\phi  \big|_{L^{{p'}} }    \\
&  \lesssim  \ep^2  \mu^{-k_0   }    \big|         \phi  \big|_{\dot{ B }^{     k_0 }_{{p'},1 }  }  .
\end{aligned}
\end{equation}
\end{proof}

Next, we show the same  upper bound for  $\overline{u}_z  \ez $ as $\overline{u}_r  \er $.

\begin{lemma}\label{lemma:est_u_0z}
Under the setting of Proposition \ref{prop:approximate_inflation_s<0}, there holds
\begin{align} \label{eq:est_u_0z} 
|  \overline{u}_z    \ez  |_{  \dot B^{s}_{p,1}      }   & \lesssim  \ep^{   2 } .
\end{align}
\end{lemma}
\begin{proof}

To establish the same bounds for $\overline{u}_z  \ez  = \overline{u}_{z,p}  \ez + \overline{u}_{z,c}  \ez $, we use the same strategy as Lemma \ref{lemma:est_u_0r}. Let   $\phi  \in  \mathcal{S} (\RR^3)$ and consider the integral
\begin{equation}\label{eq:aux_lemma:est_u_0z_1} 
\int_{\RR^3} \overline{u}_z  \ez \cdot \phi \, dx .
\end{equation}

For $\overline{u}_{z,p}\ez $, we first   switch to cylindrical coordinates and integrate by parts
\begin{equation}\label{eq:aux_lemma:est_u_0z_2} 
\begin{aligned} 
& \int_{\RR^3} \overline{u}_{z,p}  \ez   \cdot \phi \, dx  =  \ep^2 \mu^{-1+ \frac{2}{p}  -s } \nu^{ \frac{1}{p}   } \int     \p_r \big(  f    (  \mu  \rho  ) \big)        \ez \cdot     \phi  \, r dr d\theta dz   \\
& = - \ep^2 \mu^{-1+ \frac{2}{p}  -s } \nu^{ \frac{1}{p}   } \int     f    (  \mu  \rho  )        \ez \cdot  \big( r  \p_r \phi + \phi \big)    \,   dr d\theta dz  .
\end{aligned}
\end{equation}
Similarly,  by the definition of $\overline{u}_{z,p}$ we obtain 
\begin{equation} \label{eq:aux_lemma:est_u_0z_3} 
\begin{aligned} 
\int_{\RR^3} \overline{u}_{z,c}  \ez   \cdot \phi \, dx & =  \ep^2 \mu^{-1+ \frac{2}{p}  -s } \nu^{ \frac{1}{p}   } \int       f    (  \mu  \rho  )         \ez \cdot     \phi  \,   dr d\theta dz   .
\end{aligned}
\end{equation}
So adding up \eqref{eq:aux_lemma:est_u_0z_1}  and \eqref{eq:aux_lemma:est_u_0z_2} , we need to estimate
\begin{equation} \label{eq:aux_lemma:est_u_0z_4} 
\begin{aligned} 
\int_{\RR^3} \overline{u}_{z }  \ez   \cdot \phi \, dx & = - \ep^2 \mu^{-1+ \frac{2}{p}  -s } \nu^{ \frac{1}{p}   } \int     f    (  \mu  \rho  )         \ez \cdot     \p_r \phi  \,  r   dr d\theta dz   \\
& : = - \ep^2 \mu^{-1+ \frac{2}{p}  -s } \nu^{ \frac{1}{p}   }  I_\phi .
\end{aligned}
\end{equation}

This integral $ I_\phi$ can be estimated identically as in Lemma \ref{lemma:est_u_0r} for $ \overline{u}_{r }  \er  $ by noting  that
\begin{equation}
\Big|  \p_\rho^{k }   \p_r \phi (x)\Big| \leq \Big|  \nabla^{k }_{r,z}   \p_r \phi (x)\Big| \leq \big|  \nabla^{k+1 }   \phi (x)\big| \quad \text{for any $k\in \NN$},
\end{equation}
and we have the same estimate as $  \overline{u}_{r }$
\begin{equation} \label{eq:aux_lemma:est_u_0z_5} 
\begin{aligned} 
\Big|  \int_{\RR^3} \overline{u}_{z }  \ez   \cdot \phi \, dx \Big|   
&\lesssim  \ep^2 \mu^{-k_0    -s } |    \phi |_{ \dot{ B }^{     k_0 }_{{p'},1 }    } .
\end{aligned}
\end{equation}

\end{proof}

\subsection{Upper bounds for \texorpdfstring{$   {u}_{0,\theta}  \et$}{u0}}

In the rest of this section, we  focus on the  azimuthal part $  {u}_{0,\theta}  \et$.

\begin{lemma}\label{lemma:est_u_0theta}
Under the setting of Proposition \ref{prop:approximate_inflation_s<0}, there holds
\begin{align}\label{eq:lemma:est_u_0theta}
|  u_{0,\theta}  \et   |_{ \dot B^{s}_{p,1}     }  & \lesssim  \ep^{   2 } . 
\end{align}
\end{lemma}
\begin{proof}

For brevity of notation, let us denote $k_1 = k_0 -2 = 4$. Since also $k_1 > |s|+1$, it suffices to estimate the $L^p$ and $ \dot{ B }^{     -k_1 }_{p , 1  } $ norm and then use Lemma \ref{lemma:besov_interpolation} to interpolate.

From the definition of $u_{0,\theta} $ \eqref{eq:def_overline_u_theta0} we obtain immediately
\begin{equation}\label{aux:lemma:est_u_0theta_1}
|  u_{0,\theta}  \et   |_{ L^{p} (\RR^3)    } \lesssim \ep^{-2} \mu^{-s}
\end{equation}
to prove \eqref{eq:lemma:est_u_0theta}, it suffices to prove
\begin{align}\label{aux:lemma:est_u_0theta_1a}
|  u_{0,\theta}  \et   |_{ \dot B^{-k_1 }_{p,1}     }  & \lesssim  \ep^{  -2 +k_1 N } \mu^{-k_1 - s} .
\end{align}
Indeed,  since $-k_1 <s <0$, by Lemma \ref{lemma:besov_interpolation}, \eqref{aux:lemma:est_u_0theta_1} and \eqref{aux:lemma:est_u_0theta_1a} imply that
\begin{align} \label{aux:lemma:est_u_0theta_2}
|  u_{0,\theta}  \et   |_{ \dot B^{s }_{p,1}     }  & \lesssim \big( \ep^{  -2  } \big)^{ 1+  \frac{s}{k_1}}   \big( \ep^{  -2 +k_1  N } \big)^{ -\frac{s}{k_1}}  \leq \ep^{ -2+     N s  } \leq \ep^{  2}  ,
\end{align}
where we have used that $N s \geq 100$.

From now on we focus on proving \eqref{aux:lemma:est_u_0theta_1a}. To this end, let   $\phi  \in  \mathcal{S} (\RR^3)$ and consider the integral
\begin{equation} 
\int_{\RR^3}u_{0,\theta}  \et  \cdot \phi \, dx .
\end{equation}

For simplicity denote by $ \zeta  (\rho ) = t^*\ep^2 \mu^{ s   +\frac{2}{p } } \nu^{ \frac{1}{p}} \rho^{-1}    =  \ep^{ -N}  \mu^{  -1 }   \rho^{-1}   $. Recalling the definition of ${u}_{0,\theta}$ from \eqref{eq:def_overline_u_theta0}
and   integrating by parts $k_1=4$ times in $\rho$, we have
\begin{align}
& \int_{\RR^3} u_{0,\theta} \et \cdot \phi \, dx \\
&   =     \ep^{-2 + k_0 N }\mu^{\frac{2}{p } -  s - k_0   } \nu^{ \frac{1}{p}}   
\int_{\RR^+ \times \TT \times \TT }  \p_\rho^{2} \left(   \sin(\varphi + \zeta  (\rho )   ) g (\mu \rho )  \right)   \p_\rho^{k_1 } \Big(  \et \cdot     \phi r \rho \Big) \, d\rho d\varphi d \theta  . \label{aux:lemma:est_u_0theta_3}
\end{align}
Using $  \p_\rho^2 (\phi r ) = 2\cos(\varphi)  $ we get  $ \p_\rho^{k_1 } \Big(  \et \cdot     \phi r \rho \Big) =     \et \cdot   \p_\rho^{k_1}    \phi r \rho  + k_1  \et \cdot   \p_\rho^{k_1-1}    \phi (r + \rho \cos(\varphi) )  + 2 C_{k_1}^2 \et \cdot   \p_\rho^{k_1-2}    \phi  \cos(\varphi) $
and so we can decompose the integral into three parts:
\begin{equation}\label{aux:lemma:est_u_0theta_4}
\int_{\RR^3} u_{0,\theta} \et \cdot \phi \, dx = \ep^{-2  + k_0 N } \mu^{\frac{2}{p } -  s - k_0   } \nu^{ \frac{1}{p}} \Big( I_1 + k_1 I_2 +2 C^2_{k_1}  I_3 \Big)
\end{equation}
with 
\begin{align}
I_1 & =  \int_{\RR^+ \times \TT \times \TT }  \p_\rho^{2} \left(   \sin(\varphi + \zeta  (\rho )   ) g (\mu \rho )  \right)   \et \cdot   \p_\rho^{k_1}    \phi r \rho \, d\rho d\varphi d \theta  \label{aux:lemma:est_u_0theta_5a}\\
I_2 & =  \int_{\RR^+ \times \TT \times \TT }  \p_\rho^{2} \left(   \sin(\varphi + \zeta  (\rho )   ) g (\mu \rho )  \right)   \et \cdot   \p_\rho^{k_1 -1}    \phi (r + \rho \cos(\varphi) )  \, d\rho d\varphi d \theta  \label{aux:lemma:est_u_0theta_5b}\\
I_3 & =  \int_{\RR^+ \times \TT \times \TT }  \p_\rho^{2} \left(   \sin(\varphi + \zeta  (\rho )   ) g (\mu \rho )  \right)   \et \cdot   \p_\rho^{k_1 -2}    \phi  \cos(\varphi)    \, d\rho d\varphi d \theta  . \label{aux:lemma:est_u_0theta_5c}
\end{align}
For $ I_1$, we can estimate directly and obtain
\begin{equation}\label{aux:lemma:est_u_0theta_6}
|I_1| \lesssim \left| \p_\rho^{2} \left(   \sin(\varphi + \zeta  (\rho )   ) g (\mu \rho )  \right)   \right|_{L^p(\RR^3)} | \p_\rho^{k_1}    \phi |_{L^{{p'}} } .
\end{equation}
The first part of \eqref{aux:lemma:est_u_0theta_6} can be estiamted by Lemma \ref{lemma:G_estimates}, yielding 
\begin{equation}\label{aux:lemma:est_u_0theta_6a}
\left| \p_\rho^{2} \left(   \sin(\varphi + \zeta  (\rho )   ) g (\mu \rho )  \right)   \right|_{L^p(\RR^3)} \lesssim (\ep^{-2 N } \mu^{2})  (\mu^{- \frac{2}{p} } \nu^{- \frac{1}{p} } )    .
\end{equation} 
And hence by  \eqref{eq:diff_rho_varphi_2}
\begin{equation}\label{aux:lemma:est_u_0theta_7}
|I_1| \lesssim (\ep^{-2 N } \mu^{2})  (\mu^{- \frac{2}{p} } \nu^{- \frac{1}{p} } )   |   \phi |_{\dot B^{k_1}_{{p'},1} } .
\end{equation}

For $I_{2}$ we integrate by part once more in $\rho$ to further split it as
\begin{align}\label{aux:lemma:est_u_0theta_8}
I_2 = -  (I_{21} + I_{22}),
\end{align}
with 
\begin{align}
I_{21}   & =   \int_{\RR^+ \times \TT \times \TT }  \p_\rho  \left(   \sin(\varphi + \zeta  (\rho )   ) g (\mu \rho )  \right)   \et \cdot   \p_\rho^{k_1  }    \phi (r + \rho \cos(\varphi) )  \, d\rho d\varphi d \theta \label{aux:lemma:est_u_0theta_9a} \\
I_{22}   & =   \int_{\RR^+ \times \TT \times \TT }  \p_\rho  \left(   \sin(\varphi + \zeta  (\rho )   ) g (\mu \rho )  \right)   \et \cdot   \p_\rho^{k_1-1  }    \phi \cos(\varphi)    \, d\rho d\varphi d \theta  . \label{aux:lemma:est_u_0theta_9b}
\end{align}
We bound $I_{21}$ by absolute value of the integrand, obtaining
\begin{equation}\label{aux:lemma:est_u_0theta_10}
\begin{aligned}
|I_{21} | & \lesssim \int_{\RR^3 } \big|\p_\rho  \left(   \sin(\varphi + \zeta  (\rho )   ) g (\mu \rho )  \right)  \big|     \big|  \p_\rho^{k_1  }    \phi \big|  (r + \rho   )  \,  \rho^{-1} r^{-1} \, d x \\
& \lesssim  \nu \left| \p_\rho  \left(   \sin(\varphi + \zeta  (\rho )   ) g (\mu \rho )  \right)      \right|_{L^p } | \p_\rho^{k_1}    \phi |_{L^{{p'}} }
\end{aligned}
\end{equation}
By Lemma \ref{lemma:G_estimates} and \eqref{eq:diff_rho_varphi_2} again, it follows that
\begin{equation}\label{aux:lemma:est_u_0theta_11}
\begin{aligned}
|I_{21} | &   \lesssim  \nu \left| \p_\rho  \left(   \sin(\varphi + \zeta  (\rho )   ) g (\mu \rho )  \right)      \right|_{L^p } |    \phi |_{\dot B^{k_1}_{{p'}, 1 } } \\
&   \lesssim  \ep^{-N} \mu^{ 1 - \frac{2}{p}}  \nu^{ 1 - \frac{1}{p}}   |    \phi |_{\dot B^{k_1}_{{p'}, 1 } } .
\end{aligned}
\end{equation}

For $I_{22}$ we integrate by parts once more and then bound it by the absolute value
\begin{align}
| I_{22} | & \leq    \Big| \int_{\RR^+ \times \TT \times \TT }       \sin(\varphi + \zeta  (\rho )   ) g (\mu \rho )     \et \cdot   \p_\rho^{k_1   }    \phi \cos(\varphi)    \, d\rho d\varphi d \theta  \Big| \\
& \lesssim   \int_{\RR^3 }        \big| g (\mu \rho )   \big|    \big|  \nabla^{k_1 }   \phi     \big| (r \rho)^{-1} \, dx . \label{aux:lemma:est_u_0theta_13}
\end{align}
It follows that
\begin{align}\label{aux:lemma:est_u_0theta_14}
| I_{22} |    \lesssim \mu \nu     \big| g (\mu \rho )   \big|_{L^p(\RR^3)}    \big|  \nabla^{k_1 }      \phi     \big|_{L^{{q}} }   \lesssim \mu^{ 1  - \frac{2}{p} } \nu^{1  - \frac{1}{p} }    |  \nabla^{k_1 }      \phi      |_{ \dot B^{k_1}_{{p'} , 1} }  .
\end{align}

Putting together \eqref{aux:lemma:est_u_0theta_11} and \eqref{aux:lemma:est_u_0theta_14} we have
\begin{equation}\label{aux:lemma:est_u_0theta_14b}
| I_{2 } |    \lesssim \ep^{-N} \mu^{  1 - \frac{2}{p}}  \nu^{ 1 - \frac{1}{p}}   |    \phi |_{\dot B^{k_1}_{{p'}, 1 } } .  
\end{equation}

Finally noting that $|I_3| \sim |I_{22}|$, we combine the estimates for $I_i$'s to obtain
\begin{equation}\label{aux:lemma:est_u_0theta_15}
\begin{aligned}
& \Big| \int_{\RR^3} u_{0,\theta} \et \cdot \phi \, dx  \Big| \\
& \lesssim   \ep^{-2  +k_0 N  } \mu^{\frac{2}{p } -  s - k_0   } \nu^{ \frac{1}{p}}   \Big(  \ep^{-2 N} \mu^{ 2 - \frac{2}{p} } \nu^{ - \frac{1}{p} } + \ep^{-N} \mu^{  1 - \frac{2}{p}}  \nu^{ 1 - \frac{1}{p}}    \Big)  |  \nabla^{k_1 }      \phi      |_{ \dot B^{k_1}_{{p'} , 1} } \\ 
& \lesssim   \ep^{-2 +k_1 N  } \mu^{  -  s - k_1   }           |  \nabla^{k_1 }      \phi      |_{ \dot B^{k_1}_{{p'} , 1} } 
\end{aligned}
\end{equation}
where in the third inequality we assumed (with loss of generality) $\ep<1$ and used  $\mu \geq \nu $.

\end{proof}

\subsection{Lower bounds for \texorpdfstring{$ \overline{u}_{ \theta} (t^*) \et $}{ut}}

Next, we show the growth of $\overline{u}_{ \theta} $ at time $t^*$. The method of proving a lower bound is reminiscent of Lemma \ref{lemma:G_estimates}.

\begin{lemma}\label{lemma:est_u_theta_t}
Under the setting of Proposition \ref{prop:approximate_inflation_s<0}, there exists $\ep_0 = \ep_0(s,p)>0$ such that if $ 0< \ep \leq \ep_0$, then  
\begin{align}\label{eq:lemma:est_u_theta_t}
|  \overline{u}_{\theta}(t^*)  \et   |_{ \dot B^{s}_{p,1}     }  & \gtrsim  \ep^{  - 2 }  .
\end{align}
\end{lemma}
\begin{proof}

Since $  s<0$,  to find the lower bound in $\dot B^{s}_{p,1} $, we can use the inverse interpolation inequality
\begin{equation}\label{eq:aux_lemma:est_u_theta_t_1}
|  \overline{u}_{ \theta }(t^*)  \et  |_{ \dot B^{s}_{p,1}     }  \gtrsim   |  \overline{u}_{ \theta }(t^*)  \et  |_{ \dot W^{1,  p}     }^{ s } |  \overline{u}_{ \theta }(t^*)  \et  |_{  L^{p }     }^{1-s} .
\end{equation}
which is a direct consequence   of Lemma \ref{lemma:besov_interpolation}.

We will estimate the right-hand side of \eqref{eq:aux_lemma:est_u_theta_t_1} in two steps below.

\noindent
\textbf{\underline{Step 1: $L^p$ bound}}

Recall that the initial data of $\overline{u}_{\theta}(t^*)$   is given  by 
\begin{equation}\label{eq:aux_lemma:est_u_theta_t_2}
\overline{u}_{0,\theta}  = \ep^{-2 + k_0 N } \mu^{\frac{2}{p } -  s - k_0  } \nu^{ \frac{1}{p}}  \p_\rho^{k_0 }  \sum_{0\leq i \leq k_0 }\p_\rho^{ i} \left(   \sin(\varphi + \zeta  (\rho )   )\right)  \mu^{ k-i} g^{(k-i)}(\mu \rho )
\end{equation}
where as before we denote  $ \zeta  (\rho ) = t^*  \ep^2 \mu^{ s   +\frac{2}{p } } \nu^{ \frac{1}{p}} \rho^{-1}       =  \ep^{ -N}  \mu^{-1} \rho^{-1}$ for simplicity.

The idea is to identify the main term in \eqref{eq:aux_lemma:est_u_theta_t_2} as in Lemma \ref{lemma:G_estimates} and    track its time evolution.

Further analyzing \eqref{eq:aux_lemma:est_u_theta_t_2}, by Faa di bruno, for any $i\in \NN$, we see that
\begin{equation}\label{eq:aux_lemma:est_u_theta_t_3}
\p_\rho^{ i} \left(   \sin(\varphi + \zeta  (\rho )  )   \right) = \sum_{m_1,\dots,m_i} C_{  m_1,\dots,m_i} \sin^{(m_1+\dots+m_i)}(\varphi + \zeta  (\rho )  ) \Pi_{ 1\leq j \leq  i }  \left( \zeta^{(j)}  \right)^{m_j},
\end{equation}
where the summation runs over $m_1 + 2m_2 + \dots + i m_i = i $.

Then by the design of $\overline{u}_{ \theta}  $ (transported by $\overline{u}_{ r}\er + \overline{u}_{z,p} \ez$), from Lemma \ref{lemma:overline_u_2D_Euler}, \eqref{eq:aux_lemma:est_u_theta_t_2} and \eqref{eq:aux_lemma:est_u_theta_t_3} we see that all the angle phase $\varphi + \zeta  (\rho )$ gets un-mixed to $\varphi$ at the time $t =t^* $. Namely, we have 
\begin{equation}\label{eq:aux_lemma:est_u_theta_t_4}
\begin{aligned}
\overline{u}_{ \theta} (t^*) \et & =   \ep^{-2 + k_0 N } \mu^{\frac{2}{p } -  s - k_0  } \nu^{ \frac{1}{p}}   \bigg(    \sum_{0\leq i \leq k_0  }   \mu^{ k-i} g^{(k-i)}(\mu \rho ) \\ 
\quad & \quad   \sum_{m_1,\dots,m_i} C_{  m_1,\dots,m_i} \sin^{(m_1+\dots+m_i)}(\varphi   ) \Pi_{ 1\leq j \leq  i }  \left( \zeta^{(j)}  \right)^{m_j}  \bigg)\et .
\end{aligned}
\end{equation}

Let us consider the main term in \eqref{eq:aux_lemma:est_u_theta_t_4}, defined by
\begin{equation}\label{eq:aux_lemma:est_u_theta_t_5}
\overline{u}_{ \theta, p } \et : =  \ep^{-2 + k_0 N } \mu^{\frac{2}{p } -  s - k_0  } \nu^{ \frac{1}{p}}       \sin^{(k_0 )}(\varphi   )  \left( \zeta^{(1)}  \right)^{k_0 }   g (\mu \rho )\et,
\end{equation}
and the decomposition
\begin{equation}\label{eq:aux_lemma:est_u_theta_t_6}
\overline{u}_{ \theta  }(t^*) \et  = \overline{u}_{ \theta, p } \et + \overline{u}_{ \theta, c } \et 
\end{equation}

As in the proof Lemma \ref{lemma:G_estimates},  on the support of $g(\mu\rho)$, we have $ |\zeta^{(j)}(\rho)| \sim  \ep^{ -N}  \mu^{ j}   $. It follows that in $ \overline{u}_{ \theta, c }  $, the  largest possible factor of $\ep$ is $(\ep^{ -N}   )^{k_0 - 1 }  $, namely
\begin{equation}\label{eq:aux_lemma:est_u_theta_t_7}
\begin{aligned}
|  \overline{u}_{ \theta, c } \et   |_{L^{ p}} &  \lesssim  \ep^{-2 + k_0 N } \mu^{\frac{2}{p } -  s - k_0  } \nu^{ \frac{1}{p}}     \left( \mu^{k_0} (\ep^{ -N}   )^{k_0 - 1 }   \right) \\
&  \lesssim  \ep^{-2  +N } \mu^{  -  s   }  .
\end{aligned}
\end{equation}
On the other hand, the main term \eqref{eq:aux_lemma:est_u_theta_t_5} satisfies  
\begin{equation}\label{eq:aux_lemma:est_u_theta_t_8}
\begin{aligned}
| \overline{u}_{ \theta, p } \et  |_{L^{ p}} & \sim \ep^{-2 + k_0 N } \mu^{\frac{2}{p } -  s - k_0  } \nu^{ \frac{1}{p}} \Big( \mu^{k_0} \ep^{-k_0 N } | g (\mu \rho )  |_{L^{ p}}    \Big)  \\
& \sim  \ep^{-2   } \mu^{  -  s   }  .       
\end{aligned}
\end{equation}
From \eqref{eq:aux_lemma:est_u_theta_t_7} and \eqref{eq:aux_lemma:est_u_theta_t_8}, it follow that for all sufficiently small $\ep$,  there holds
\begin{equation}\label{eq:aux_lemma:est_u_theta_t_9}
\begin{aligned}
| \overline{u}_{ \theta } \et  |_{L^{ p}} & \sim \ep^{-2   }  \mu^{  -  s   }        .
\end{aligned}
\end{equation}

\noindent
\textbf{\underline{Step 2: $\dot W^{1,p}$ bound}} 

Next, we establish the following upper bound for $  |  \overline{u}_{ \theta } (t^*)\et |_{\dot W^{1,p}}  $:
\begin{equation}\label{eq:aux_lemma:est_u_theta_t_10}
\begin{aligned}
|  \overline{u}_{ \theta } (t^*)\et |_{\dot W^{1,p}}  & \lesssim \ep^{-2   }  \mu^{ 1 -  s   }        .
\end{aligned}
\end{equation}    
Since $ |\nabla \et | \lesssim \nu $ on the support of $ \overline{u}_{ \theta }  $, it then suffices to show 
\begin{equation}\label{eq:aux_lemma:est_u_theta_t_11}
\begin{aligned}
|  \overline{u}_{ \theta } (t^*)  |_{\dot W^{1,p}}  & \lesssim \ep^{-2   }  \mu^{ 1 -  s   }        .
\end{aligned}
\end{equation} 
For $\nabla \overline{u}_{ \theta } (t^*)$, we need to bound $\p_\rho \overline{u}_{ \theta}$ and $\rho^{-1} \p_\varphi \overline{u}_{ \theta}$. Both    $\p_\rho \overline{u}_{ \theta}$ and $\rho^{-1} \p_\varphi \overline{u}_{ \theta}$ can be estimated in the same way as in Step 1 using \eqref{eq:aux_lemma:est_u_theta_t_4}. The same argument yields the bound  
\begin{equation}\label{eq:aux_lemma:est_u_theta_t_12}
\begin{aligned}
| \nabla \overline{u}_{ \theta} (t^*) |_{L^{p}} & \lesssim    \ep^{-2     } \mu^{1 -  s   }  .
\end{aligned}
\end{equation}
Thus \eqref{eq:aux_lemma:est_u_theta_t_10} holds. 

\noindent
\textbf{\underline{Step 3: Conclusion}} 

Now that we have \eqref{eq:aux_lemma:est_u_theta_t_9} and \eqref{eq:aux_lemma:est_u_theta_t_10}, we can go back to \eqref{eq:aux_lemma:est_u_theta_t_1} and obtain
\begin{equation} 
|  \overline{u}_{ \theta }(t^*)  \et  |_{ \dot B^{s}_{p,1}     }  \gtrsim   |  \overline{u}_{ \theta }(t^*)  \et  |_{ \dot W^{1,  p}     }^{ s } |  \overline{u}_{ \theta }(t^*)  \et  |_{  L^{p }     }^{1-s} \gtrsim     \ep^{  - 2 }  .
\end{equation}

\end{proof}

\section{Stability of approximation}\label{sec:no_blowup}

In this section, we prove that the approximate solution $ \overline{ u} $ defined in the previous section stays close in strong norms to the actual solution $u$ of \eqref{eq:NS} up to the critical time $t^*$.

In the following,  only the constants with an   $ \ep $ subscript depend on $\ep$. All constants are independent of $\mu,\nu     $ but may change from line to line.

\subsection{The main proposition}

We will be proving the following proposition, the main result of this section.
\begin{proposition}\label{prop:no_blowup}

Let $T=T(\ep, \mu  )> 0 $ be the maximal time of existence for the local-in-time solution $u: [0 , T ) \times \RR^3 \to \RR^3$ of \eqref{eq:NS} with the same initial data $u_0$ as $\overline{u}$.

For any $\ep>0$, there exists $\mu_\ep>0$ sufficiently large such that if $\mu \geq \mu_\ep$, then there must be $ 0<   t^*< T$, namely $ u  \in C^\infty( [0,t^* ] \times \RR^3 )$.

More quantitatively, for any $\ep>0$, if $\mu \geq \mu_\ep$, then for any $k \geq 0$  and $1 \leq {q} \leq \infty$
\begin{equation}\label{eq:prop_no_blowup}
|u - \overline{u}|_{L^\infty ([0,t^*]; \dot B^{k}_{q,\infty }   ) } \leq C_{k,\ep } \big(\mu^{-1}  \nu  \big)^{ \gamma  }   \mu^{k - s} \mu^{\frac{2}{ p} - \frac{2}{ {q}} }\nu^{\frac{1}{ p} - \frac{1}{ {q} } }
\end{equation}
where   $C_{k,\ep }$ is independent of $\mu$ and $\gamma  = \frac{9}{10}$ if $1 \leq  q \leq 2$ and $\gamma  = \frac{4}{10}$ if $ 2< q \leq \infty $.
\end{proposition}

The proof of Proposition \ref{prop:no_blowup}  occupies the rest of this section. The slightly more precise  \eqref{eq:prop_no_blowup} is to leave some room for the lossy embedding that will incur in Section \ref{sec:proof}.

\subsection{The bootstrap assumption}

We use a bootstrap argument that transfers the smallness of $\overline{E}$ to the difference $u - \overline{u} $ on the interval $[0,t^*]$, thus proving  Proposition \ref{prop:no_blowup}.

Denote by $w  = u - \overline{u}$ the difference between the approximate solution $ \overline{u}$ and the exact solution $u$ (having the same initial data $u_0$ as $ \overline{u} $) . It follows from Proposition \ref{prop:approximate} that the vector field $w$, well-defined on the time interval $[0,T)$,  satisfies the evolution equation
\begin{equation}\label{eq:aux_noblowup_0}
\begin{cases}
\p_t w + u\cdot \nabla w + w \cdot \nabla  \overline{u} + \nabla q = \overline{E} & \\
\D w  = 0 & \\
w |_{t = 0 } = 0 &
\end{cases}
\quad \text{for $(t,x) \in [0,T ) \times \RR^3 $}
\end{equation}
where $q : = p - \overline{p}$ with $p$ being the pressure of the exact solution $u$ and $\overline{p}$ the pressure from Proposition \ref{prop:approximate}.

For any given $\ep>0$, let us fix $M_\ep \geq 1$  such that 
\begin{equation}\label{eq:aux_noblowup_1}
|\nabla \overline{u}  |_{L^\infty(  [ 0,t^*] \times \RR^3) } \leq M_\ep \mu^{1- s +\frac{2}{p} } \nu^{\frac{1}{p}}   .
\end{equation}  
This is always possible thanks to \eqref{eq:prop_error_1} from Proposition \ref{prop:approximate}. We note that the exact value $M_\ep \geq 1$ plays no role.

For this fixed  $M_\ep$, we will prove the following bound on the exact solution $u$
\begin{equation}\label{eq:aux_noblowup_claim_0}
|\nabla u   |_{L^\infty(  [ 0,t^*] \times \RR^3) } \leq 2   M_\ep  \mu^{1- s +\frac{2}{p} } \nu^{\frac{1}{p}}   \quad \text{for all} \quad    0\leq t    \leq t^* . 
\end{equation}  
We will prove \eqref{eq:aux_noblowup_claim_0} by a bootstrap argument. Since   $|\nabla  {u}_0  |_{L^\infty} \leq M_\ep \mu^{1- s +\frac{2}{p} } \nu^{\frac{1}{p}} $ holds at $t = 0$, by  continuity  let us introduce

\mdfsetup{skipabove=2pt,skipbelow=2pt}
\begin{mdframed}[linewidth=1pt,frametitle={The bootstrap assumption:},nobreak=true]

For some $ 0< t_0 \leq t^*$, there holds 
\begin{equation}\tag{$\dagger$}\label{eq:aux_noblowup_claim}
|\nabla u (t) |_{L^\infty } \leq 2   M_\ep  \mu^{1- s +\frac{2}{p} } \nu^{\frac{1}{p}}   \quad \text{for all} \quad    0\leq t    \leq t_0  .
\end{equation}

\end{mdframed}

In the steps below, we will prove that if $\mu $ is sufficiently large (depending on $\ep>0$), then under the bootstrap assumption \eqref{eq:aux_noblowup_claim},  we  have the improved bound
$$
|\nabla u (t) |_{L^\infty } \leq    \frac{3}{2} M_\ep  \mu^{1- s +\frac{2}{p} } \nu^{\frac{1}{p}}  \quad \text{for all} \quad    0\leq t    \leq t_0  ,
$$
which will imply   \eqref{eq:aux_noblowup_claim} on the whole interval $[0, t^*]$ by continuity.

\subsection{Basic estimates}

To facilitate the estimates, we will frequently use the following bounds which are  direct consequences of Proposition \ref{prop:kato_ponce},  the definition of $t^*$ (from \eqref{eq:def_critical_t*}), and the bootstrap assumption \eqref{eq:aux_noblowup_claim},
\begin{equation}\label{eq:aux_noblowup_assumption_0}
\begin{cases}
| u  |_{ L^\infty([0,t_0]; W^{k, {q}} )}   & \leq C_{\ep, k,{q} }     \mu^{ k -s  } \mu^{ \frac{2}{p }  - \frac{2}{{q}}  } \nu^{\frac{1}{p} - \frac{1}{{q}}  }  ,   \\
t^* \mu^{1- s +\frac{2}{p} } \nu^{\frac{1}{p}}   & \leq C_\ep 
\end{cases} 
\end{equation}
where $k \geq 0$  and $ 1<{q}<\infty $. 

Note that in Proposition \ref{prop:no_blowup} we claimed the estimates for all $1 \leq q \leq \infty$ compared to $ 1<{q}<\infty $ in \eqref{eq:aux_noblowup_assumption_0}.

\subsection{Energy estimates}
As the first step in the bootstrap argument, we derive suitable energy estimates on $w$, gaining a small factor $  (\mu^{-1}   \nu  )   $   comparing to $   u $ and $    \overline{ u } $.

As in \cite{2404.07813}, we need the power of the smallness factor $\mu^{-1} \nu$ in  \eqref{eq:lemma:bootstrap}   strictly bigger than $\frac{1}{2}$ to compensate the lossy embedding $H^{\frac{5}{2} + } \hookrightarrow W^{1,\infty } $ in our anisotropic setup.

\begin{lemma}\label{lemma:bootstrap}
Under the bootstrap assumption \eqref{eq:aux_noblowup_claim}, the difference $w = u - \overline{u}$ satisfies for any integer $k \geq 0$ the estimate,
\begin{equation}\label{eq:lemma:bootstrap}
\begin{aligned}
|\nabla^k w(t) |_{L^{2}}  \leq  C_{\ep, k } \big(\mu^{-1}  \nu  \big)^{\delta_k}     \mu^{k - s } \mu^{ \frac{2}{p }  - 1 } \nu^{\frac{1}{p} - \frac{1}{2}  } \quad \text{for any $t \in [0, t_0 ]$}
\end{aligned}
\end{equation}
where $\delta_k  =   \frac{ 9  + 10^{- k  }  }{10}        > 0$ for each $ k \in \NN $.

\end{lemma}
\begin{proof}

For each $k \in \NN$ we need to find $\mu$-independent constants $C_{\ep,k}$ such that \eqref{eq:lemma:bootstrap} holds. We will prove by induction.

\noindent
\textbf{\underline{Step 1: The case $k =0  $}}

Since both $u$ and  $\overline{u} $ are divergence-free, multiplying \eqref{eq:aux_noblowup_0} by $ w = u - \overline{u}$ and integrating, we obtain 
\begin{equation}\label{eq:aux_lemma:bootstrap_1}
\frac{d}{dt}|w(t) |_{L^{2}}^2 + |\nabla w|_{L^{2}}^2  \lesssim   |\nabla \overline{u}  |_{L^\infty } | w |_{ L^{2} }^2     +  | E |_{L^{2}}| w |_{L^{2}}.
\end{equation} 

We drop the positive dissipation term on the left-hand side of \eqref{eq:aux_lemma:bootstrap_1}. By the estimates in Proposition \ref{prop:approximate} and \eqref{eq:aux_noblowup_assumption_0} which hold on $[0,t^*]$,  it follows that
\begin{equation}\label{eq:aux_lemma:bootstrap_2}
\begin{aligned}
\frac{d}{dt}|w(t) |_{L^{2}}  & \lesssim_\ep \mu^{1 -s+ \frac{2}{p}} \nu^{\frac{1}{p}} | w |_{ L^{2} }      +\mu^{   -   s}   (\mu^{-1}   \nu )  \mu^{1 -s+ \frac{2}{p}} \nu^{\frac{1}{p}} \mu^{ \frac{2}{p }  - 1 } \nu^{\frac{1}{p} - \frac{1}{2}  }   .
\end{aligned} 
\end{equation}

Applying Gronwall's inequality (recalling $w |_{t=0} = 0 $) to \eqref{eq:aux_lemma:bootstrap_2} yields 
\begin{equation}\label{eq:aux_lemma:bootstrap_3}
\begin{aligned}
|w(t) |_{L^{2}}   & \lesssim_\ep   e^{C_\ep t \mu^{1 -s+ \frac{2}{p}} \nu^{\frac{1}{p}}  }   \Big(   t \mu^{   -   s}   (\mu^{-1}   \nu )  \mu^{1 -s+ \frac{2}{p}} \nu^{\frac{1}{p}} \mu^{ \frac{2}{p }  - 1 } \nu^{\frac{1}{p} - \frac{1}{2}  }  \Big)            \quad \text{for any $t \in [0, t_0 ]$}  .
\end{aligned}
\end{equation}

Then by \eqref{eq:aux_noblowup_assumption_0}, the exponential factor in \eqref{eq:aux_lemma:bootstrap_3} is independent of $\mu$ and $\nu$, and hence we have shown \eqref{eq:lemma:bootstrap} when $k=0$:
\begin{equation}\label{eq:aux_lemma:bootstrap_4}
|w(t) |_{L^{2}}   \leq C_\ep      \big(\mu^{-1}  \nu  \big)   \mu^{ -s + \frac{2}{p }  - 1 } \nu^{\frac{1}{p} - \frac{1}{2}  }   \quad \text{for any $t \in [0, t_0 ]$} .
\end{equation}

\noindent
\textbf{\underline{Step 2: The case $k\geq 1$}}

Now we assume \eqref{eq:lemma:bootstrap} has been proved for levels $\leq k-1$ with $\mu,\nu$ independent constants.

By testing \eqref{eq:aux_noblowup_0} with $-\p_\alpha^2 w$ and summing over all multi-indexes $|\alpha| = k$, we have the following standard energy estimates for $\nabla^k  w  $:
\begin{equation} \label{eq:aux_lemma:bootstrap_5}
\begin{aligned}
\frac{d}{dt} |  \nabla^k  w  |_{L^{2}}^2  & \lesssim_k  I(t) + J(t)
\end{aligned}
\end{equation}
where we have  dropped the positive linear term, and the  terms  $I(t) , J(t) $ are given by  
\begin{equation}\label{eq:aux_lemma:bootstrap_6}
\begin{aligned}
& I(t) =    | \nabla   \overline{u}  |_{L^{\infty }}   |  \nabla^k w  |_{L^{2}}^2          + |   \nabla u |_{L^{\infty }}  |  \nabla^k w  |_{L^{2}}^2   + | \nabla^k E |_{L^{2}}|\nabla^k w |_{L^{2}} 
\end{aligned}
\end{equation}
and  
\begin{equation}\label{eq:aux_lemma:bootstrap_7}
\begin{aligned}
&   J (t) =   |  \nabla^k w  |_{L^{2}} \sum_{ 0 \leq   m  \leq k -1 }    |\nabla^m  w   |_{L^{2}}  \big(   |   \nabla^{k +1 - m }     \overline{u}  |_{L^\infty}  +|   \nabla^{k+1 - m }      {u}  |_{L^\infty}    \big) .
\end{aligned}
\end{equation}

The idea is that $I(t)$ can be bounded similarly to the $L^2$ case and $J(t)$ can be estimated using the induction hypothesis on $|  \nabla^m w  |_{L^{2}}$, $m \leq k-1$.

\noindent
\textbf{\underline{Step 3: Estimates of $I(t) $}}

For the term $I$, using \eqref{eq:aux_noblowup_assumption_0} and Proposition \ref{prop:approximate} we have
\begin{equation}\label{eq:aux_lemma:bootstrap_8}
I \leq  C_{\ep, k }      \mu^{1+ \frac{2}{p} -s} \nu^\frac{1}{p}  |  \nabla^k w  |_{L^{2}}^2     + | \nabla^k E |_{L^{2}}|\nabla^k w |_{L^{2}}  .
\end{equation}
For the second term in \eqref{eq:aux_lemma:bootstrap_8}, we use Young's inequality and Proposition \ref{prop:approximate}, obtaining
\begin{equation}\label{eq:aux_lemma:bootstrap_9}
\begin{aligned}
&| \nabla^k E |_{L^{2}}|\nabla^k w |_{L^{2}} \\ & \lesssim_{\ep,k }   \mu^{1+ \frac{2}{p} -s}  \nu^\frac{1}{p}   |\nabla^k w |_{L^{2}}^2  + (  \mu^{1+ \frac{2}{p} -s}   \nu^\frac{1}{p}  )^{-1} | \nabla^k E |_{L^{2}}^2 \\
&  \lesssim_{\ep,k }  \mu^{1+ \frac{2}{p} -s} \nu^\frac{1}{p}  |\nabla^k w |_{L^{2}}^2  +  (\mu^{-1} \nu)^2   (\mu^{1+\frac{ 2}{p}-  s}    \nu^\frac{1}{p}  )  (  \mu^{k  - s+ \frac{ 2}{p}  - 1 } \nu^{\frac{1}{p} - \frac{1}{2} } )^2 .     
\end{aligned}
\end{equation}
From \eqref{eq:aux_lemma:bootstrap_8} and \eqref{eq:aux_lemma:bootstrap_9}, it follows that
\begin{equation}\label{eq:aux_lemma:bootstrap_9a}
I \lesssim_{\ep, k }\mu^{1+ \frac{2}{p} -s} \nu^\frac{1}{p}  |\nabla^k w |_{L^{2}}^2  +   (\mu^{-1} \nu)^2   (\mu^{1+\frac{ 2}{p}-  s}    \nu^\frac{1}{p}  )  (\mu^{ k  - s+ \frac{ 2}{p}  - 1 } \nu^{\frac{1}{p} - \frac{1}{2} } )^2   .
\end{equation}

\noindent
\textbf{\underline{Step 4: Estimates of $J(t)$}}

To estimate $J$, we first bound the factor $   |   \nabla^{ m }      {u}  |_{L^\infty}      $  for  $  m \in \NN   $.

By the Sobolev embedding $ W^{ m+ \frac{3}{q} +\delta, q} (\RR^3) \hookrightarrow W^{m,\infty }(\RR^3)$  for any $\delta >0$ and $q<\infty$, it follows from \eqref{eq:aux_noblowup_assumption_0} that for any $ m  \in \NN$,
\begin{equation}\label{eq:aux_lemma:bootstrap_10}
\begin{aligned}
| \nabla^m u   |_{ L^\infty( [0,t_0] \times \RR^3  ) }        & \lesssim_{\delta, p } |  u  |_{ L^\infty( [0,t_0] ;W^{m+\frac{3}{q} +\delta, p}  ) }     \\
& \lesssim_{m,\delta, q }  \mu^{ m   -s     + \frac{2}{q}     } \nu^{  \frac{1}{q } }   \Big( \mu^{\delta + \frac{ 1}{q}}  \nu^{- \frac{ 1}{q}}   \Big)     .
\end{aligned}
\end{equation} 

By \eqref{eq:def_small_b}, we  may choose $q  <\infty, \delta >0$ depending on $k$ such that
\begin{equation}\label{eq:aux_lemma:bootstrap_11}
\mu^{  \delta   } \mu^{    \frac{1 }{q}     } \nu^{  - \frac{1}{q} } \leq   \mu^{  b 10^{-k } } \leq   ( \mu^{-1} \nu )^{- 10^{-k-1 } } .
\end{equation}
It follows from \eqref{eq:aux_lemma:bootstrap_10} and \eqref{eq:aux_lemma:bootstrap_11} that  for any $m\in \NN$
\begin{equation}\label{eq:aux_lemma:bootstrap_12}
\begin{aligned}
| \nabla^m u   |_{ L^\infty( [0,t_0] \times \RR^3  ) }   \lesssim_{\ep, m,k }    (\mu^{ m   -s     + \frac{2}{p}     } \nu^{  \frac{1}{p} }  ) ( \mu^{-1} \nu )^{- 10^{-k -1} }  
\end{aligned}
\end{equation} 
where we note the loss of a very small power of $   \mu^{-1} \nu  $  compared to $|\nabla^m \overline{u}   (t)|_{ L^\infty }$.

It then follows from  \eqref{eq:aux_lemma:bootstrap_12}   and the inductive assumption at levels $ m \leq  k-1$ that 
\begin{align}
J (t)  
&  \leq  
C_{\ep, k }|  \nabla^k w  |_{L^{2}}  \sum_{    m  \leq k -1 }    (\mu^{-1}  \nu )^{\delta_{ m }  }  (\mu^{ m   -s     + \frac{2}{p}     } \nu^{  \frac{1}{p} }  ) \Big(   (\mu^{ k+1 - m   -s     + \frac{2}{p}     } \nu^{  \frac{1}{p} }  )           (\mu^{-1}  \nu )^{  - 10^{-k -1}   }  \Big)  \\
&  \leq  
C_{\ep, k }  |  \nabla^k w  |_{L^{2}} (\mu^{ k   -s     + \frac{2}{p}     } \nu^{  \frac{1}{p} }  )   (\mu^{ 1   -s     + \frac{2}{p}     } \nu^{  \frac{1}{p} }  )  \sum_{    m  \leq k -1 }  (\mu^{-1}  \nu )^{\delta_{ m }  }                 (\mu^{-1}  \nu )^{  - 10^{-k -1}   }  \\
&  \leq  
C_{\ep, k }  |  \nabla^k w  |_{L^{2}} (\mu^{ k   -s     + \frac{2}{p}     } \nu^{  \frac{1}{p} }  )   (\mu^{ 1   -s     + \frac{2}{p}     } \nu^{  \frac{1}{p} }  )     (\mu^{-1}  \nu )^{\delta_{ k }  } \label{eq:aux_lemma:bootstrap_13}.
\end{align}
where we have used that $ \delta_{ m }  - 10^{-k -1}  \geq \delta_{ k } $ for $m\leq k -1 $ in the last line.

By Young's inequality again,  we obtain from  \eqref{eq:aux_lemma:bootstrap_13} the desired estimate:
\begin{equation}\label{eq:aux_lemma:bootstrap_15}
\begin{aligned}
J (t)\lesssim_{\ep,k }   &     \mu^{ 1   -s     + \frac{2}{p}     } \nu^{  \frac{1}{p} }           |  \nabla^k w   |_{L^{2}}^2  + (\mu^{ 1   -s     + \frac{2}{p}     } \nu^{  \frac{1}{p} } )  \big[   \mu^{ k   -s     + \frac{2}{p}     } \nu^{  \frac{1}{p} }  \big]^2 (\mu^{-1}  \nu )^{ 2\delta_{k }  }    .
\end{aligned}
\end{equation}

\noindent
\textbf{\underline{Step 5: Conclusion}}

Collecting the estimates \eqref{eq:aux_lemma:bootstrap_9a} and \eqref{eq:aux_lemma:bootstrap_15} of $I,J $,  at level $k$, on the interval $ [0,t_0] $,  we have the    differential inequality
\begin{equation}\label{eq:aux_lemma:bootstrap_16}
\begin{aligned}
\frac{d}{dt} |  \nabla^k  w  |_{L^{2}}^2  & \lesssim_{\ep, k }    \mu^{ 1   -s     + \frac{2}{p}     } \nu^{  \frac{1}{p} }           |  \nabla^k w   |_{L^{2}}^2      \\
& \qquad + (\mu^{ 1   -s     + \frac{2}{p}     } \nu^{  \frac{1}{p} } )  \big[   \mu^{ k   -s     + \frac{2}{p}     } \nu^{  \frac{1}{p} }  \big]^2 (\mu^{-1}  \nu )^{ 2\delta_{k }  }          .
\end{aligned}
\end{equation}

By Gronwall's inequality, we obtain  for all $0\leq t    \leq t_0 $ the estimate
\begin{equation}\label{eq:aux_lemma:bootstrap_17}
\begin{aligned}
|  \nabla^k  w (t) |_{L^{2}}^2  & \leq C_{\ep , k  }  e^{C_\ep t \mu^{ 1   -s     + \frac{2}{p}     } \nu^{  \frac{1}{p} }  }  \Big(t        \big(\mu^{ k   -s     + \frac{2}{p}     } \nu^{  \frac{1}{p} }\big)^2  \big( \mu^{ 1   -s     + \frac{2}{p}     } \nu^{  \frac{1}{p} }   \big)    (\mu^{-1}  \nu )^{ 2 \delta_{k } }  \Big)  .  
\end{aligned} 
\end{equation}
By \eqref{eq:aux_noblowup_assumption_0} again, the exponential factor in \eqref{eq:aux_lemma:bootstrap_17} is independent of $\mu,\nu$, and we have
\begin{equation}\label{eq:aux_lemma:bootstrap_18}
\begin{aligned}
|  \nabla^k  w (t) |_{L^{2}}^2  
& \leq  C_{\ep , k  }          \big(\mu^{ k   -s     + \frac{2}{p}     } \nu^{  \frac{1}{p} }\big)^2   (\mu^{-1}  \nu )^{ 2 \delta_{k } } 
\end{aligned}
\quad \text{for all}\,\,  0\leq t    \leq t_0 .
\end{equation}

Once the induction is complete, we conclude that  \eqref{eq:lemma:bootstrap}  holds for all integer $k \in \NN$.

\end{proof}

\subsection{Proof of main proposition}

With all the preparations in hand, we can finish the

\begin{proof}[Proof of Proposition \ref{prop:no_blowup}]

We first show the smoothness of $u$ up to $t =t^*$ and then derive the bound \eqref{eq:prop_no_blowup}.

\noindent
\textbf{\underline{Step 1: Continuation of \eqref{eq:aux_noblowup_claim}}}

For $k\in \NN$, consider the family of Sobolev embedding
\begin{equation}\label{eq:aux_noblowup_9}
|  w    |_{L^\infty }  \lesssim_k     |   w  |_{H^{ \frac{ 3}{2}+   b {10}^{-k-1} }  }   
\end{equation}
where $b$ is as in \eqref{eq:def_small_b}.

Assume that the bootstrap assumption \eqref{eq:aux_noblowup_claim} is satisfied for some $0 < t_0 <t^*$. We aim to use \eqref{eq:aux_noblowup_9} and Lemma \ref{lemma:bootstrap} to show that the bootstrap assumption \eqref{eq:aux_noblowup_claim} holds up to    $t= t^*$.

By standard interpolations, the integer Sobolev  estimates in Lemma \ref{lemma:bootstrap} implies the same non-integer Sobolev estimate for $ |   w  |_{H^{ \frac{ 3}{2}+   b {10}^{-k-1} }  }    $.  It follows     that 
\begin{equation}\label{eq:aux_noblowup_90}
\begin{aligned}
| \nabla^k w (t)  |_{L^\infty ([0,t_0] \times \RR^3 )} & \lesssim_{\ep,k}      (\mu^{-1} \nu )^{  \delta_k } \big(\mu^{ k  +\frac{ 3}{2}+   b{10}^{-k-1} -s      }\mu^{         \frac{2}{p} -1     } \nu^{  \frac{1}{p} - \frac{1}{2}}\big) \\
&   \lesssim_{\ep,k}      (\mu^{-1} \nu )^{ \delta_k -   \frac{1}{2}  - {10}^{-k-1}  } \big(\mu^{ k      -s      }\mu^{         \frac{2}{p}      } \nu^{  \frac{1}{p} }\big) \\
&   \lesssim_{\ep,k}      (\mu^{-1} \nu )^{ \frac{2}{5} }\big(\mu^{ k      -s      }\mu^{         \frac{2}{p}      } \nu^{  \frac{1}{p} }\big)  .
\end{aligned}
\end{equation}
where we have used the definition of $\delta_k$. 

In particular, from \eqref{eq:aux_noblowup_90} we have
\begin{equation}\label{eq:aux_noblowup_9a}
\begin{aligned}
| \nabla w (t)   |_{L^\infty ([0,t_0] \times \RR^3 )}  &  
& \leq C_{\ep}  \mu^{- \frac{2b}{5}  }   \mu^{1   -s +\frac{2}{p} } \nu^{\frac{1}{p}  } .
\end{aligned}
\end{equation}
Thanks to the negative exponent $- \frac{2b}{5}$ on $\mu$ in \eqref{eq:aux_noblowup_9a}, by taking $\mu \geq 1$ sufficiently large depending on $\ep$, we can ensure that 
\begin{equation}\label{eq:aux_noblowup_9b}
| \nabla w (t)  |_{L^\infty } \leq \frac{1 }{2}  M_\ep  \mu^{2  -s } \nu^{\frac{1}{2}  }  \quad \text{for any $t \in [0,t_0]$} 
\end{equation}
where $M_\ep \geq 1$ is the same constant from the bootstrap assumption \eqref{eq:aux_noblowup_claim}.

Combining \eqref{eq:aux_noblowup_1} and \eqref{eq:aux_noblowup_9b}, for any $t \in [0,t_0]$ we have  
\begin{equation}\label{eq:aux_noblowup_9c}
\begin{aligned}
| \nabla u (t)  |_{L^\infty } & \leq | \nabla  \overline{u}   (t)  |_{L^\infty } +  | \nabla w (t)  |_{L^\infty } \\
& \leq  \frac{3}{2} M_\ep  \mu^{2  -s } \nu^{\frac{1}{2}  }  .
\end{aligned}
\end{equation}

In other words, we have shown that for any $\ep>0$, there exists $\mu_\ep>0$ such that if $\mu \geq \mu_\ep$ then under the bootstrap assumption \eqref{eq:aux_noblowup_claim} with $0< t_0 < t^*$, the improved bound \eqref{eq:aux_noblowup_9c} holds. From here we conclude that  
\begin{equation}\label{eq:aux_noblowup_9ca}
| \nabla u (t)  |_{L^\infty } \leq \frac{3}{2} M_\ep  \mu^{2  -s } \nu^{\frac{1}{2}  }   \quad \text{for any $t \in [0,t^*]$} 
\end{equation}
and in particular $u$ does not blowup at $ t = t^*$ and $T> t^*$.

\noindent
\textbf{\underline{Step 2: Verification of \eqref{eq:prop_no_blowup}}}

Finally, we show the bound \eqref{eq:prop_no_blowup}.

Observe that \eqref{eq:aux_noblowup_90} implies \eqref{eq:prop_no_blowup}   for $2 \leq q \leq \infty$ by a standard interpolation with Lemma \ref{lemma:bootstrap}. It suffices to consider $k\in \NN$ with $1 \leq  {q} < 2$.

We use the integral equation
\begin{equation} \label{eq:aux_noblowup_12}
w(t) = \int_0^t e^{(t-\tau)\Delta}  \left(  - \mathbb{P}\D( w\otimes u + \overline{u}\otimes w )  + \overline{E}\right) \, d \tau 
\end{equation}
where $\mathbb{P}$ denotes the projection onto divergence-free vector fields.

Since $\mathbb{P}$ and $e^{ t \Delta}$  are bounded  on $\dot B^{k}_{q,\infty}$, we need to bound 
\begin{equation}\label{eq:aux_noblowup_13}
\begin{aligned}
|w |_{L^\infty([0,t^*]; \dot B^{k}_{q,\infty}  )} \lesssim &     t^*|w\otimes u + \overline{u}\otimes w|_{L^\infty([0,t^*]; \dot B^{k+1}_{q,\infty}   ) }  +   t^*|\overline{E}  |_{L^\infty([0,t^*]; \dot B^{k}_{q,\infty}  )} .
\end{aligned}
\end{equation}

We estimate  the right hand-side of \eqref{eq:aux_noblowup_13} in   order.

For the nonlinear part, we first work with $ \dot W^{k+1 ,{q}} $  whose norm dominates that of $\dot B^{k+1}_{q,\infty}$. Since $k\in \NN$, by product rule we have   
\begin{equation}\label{eq:aux_noblowup_14}
t^*|w\otimes u + \overline{u}\otimes w|_{L^\infty([0,t^*]; \dot W^{k+1 ,{q}})} \lesssim t^* \sum_{ 0 \leq i \leq k+1 } |\nabla^{k+1 - i} w|_{L^2} \Big(   |\nabla^{i} \overline{u} |_{L^r}  +|\nabla^{  i}   {u} |_{L^r}   \Big) .
\end{equation}
where $ r: = \frac{2{q}}{2 - {q}} \in  [  2,\infty)$.

By Proposition \ref{prop:approximate},     \eqref{eq:aux_noblowup_assumption_0}, and Lemma \ref{lemma:bootstrap}, it follows from \eqref{eq:aux_noblowup_14} that
\begin{equation}\label{eq:aux_noblowup_15}
\begin{aligned}
& t^*|w\otimes u + \overline{u}\otimes w|_{L^\infty([0,t^*]; B^{k+1}_{q,\infty} )} \\
& \lesssim_{\ep,k,q } t^*  (\mu^{-1} \nu)^{\delta_k  } \mu^{ k+1    } \big( \mu^{ -s+ \frac{2}{p }  - 1   } \nu^{\frac{1}{p} - \frac{1}{2}  }\big) \big( \mu^{ -s+ \frac{2}{p }  - \frac{2}{r }   } \nu^{\frac{1}{p} - \frac{1}{r}  }\big)  \\
& \lesssim_{\ep,k,q }   (\mu^{-1} \nu)^{  \frac{9}{10}}  \mu^{ k-s+       \frac{2}{p }  - \frac{2}{q}  } \nu^{\frac{1}{p} - \frac{1}{q}  } 
\end{aligned}
\end{equation}
where we have used $ r  = \frac{2{q}}{2 - {q}} $ and $\delta_k > \frac{9}{10}$.

Finally, for the source term by Proposition \ref{prop:approximate} we have
\begin{equation}\label{eq:aux_noblowup_17}
t^*|\overline{E}  |_{L^\infty([0,t^*]; \dot B^{k }_{q,\infty} )} \lesssim_{\ep,k  }  (\mu^{-1} \nu)^{     1 }    \mu^{ k-s+       \frac{2}{p }  - \frac{2}{q}  } \nu^{\frac{1}{p} - \frac{1}{q}  } .
\end{equation}

Combining \eqref{eq:aux_noblowup_13}, \eqref{eq:aux_noblowup_15}, and \eqref{eq:aux_noblowup_17}, we have for any $1 \leq q < 2$
\begin{equation}
|w |_{L^\infty([0,t^*]; \dot B^{k }_{q,\infty} )} \lesssim_{\ep,k,q }  (\mu^{-1} \nu)^{    \frac{9}{10 }  }    \mu^{ k-s+       \frac{2}{p }  - \frac{2}{q}  } \nu^{\frac{1}{p} - \frac{1}{q}  }  
\end{equation}
which shows \eqref{eq:prop_no_blowup} with $1\leq  q< 2$.

\end{proof}

\section{Proof of main Theorems}\label{sec:proof}

In this last section, we finish the proof of Theorem \ref{thm:Besov}.

\subsection{Proof of Theorem \ref{thm:Besov}}

\begin{proof}[Proof of Theorem \ref{thm:Besov} for $s > 0$]

In this case all statements in Theorem \ref{thm:Besov}   have been proved in the previous sections.

Indeed, given such $s>0$ and $p$ we first take $\ep>0$ smaller if necessary so that  Proposition \ref{prop:approximate_inflation_s>0} holds. Then we have:
\begin{itemize}
\item Smallness of the initial data \eqref{eq:thm_Hs_NS_1}--- proved  in Proposition \ref{prop:approximate_inflation_s>0}.

\item Smoothness of the local solution $u$ on $[0,t^*]$--- proved in Proposition \ref{prop:no_blowup}.

\item Growth of $| u |_{ \dot{B}^{s}_{p, \infty  } }  $, i.e.  \eqref{eq:thm_Hs_NS_2}---- follows from \eqref{eq:prop:approximate_inflation_s>0_2} in Proposition \ref{prop:approximate_inflation_s>0} and \eqref{eq:prop_no_blowup} in Proposition \ref{prop:no_blowup} by taking $\mu$ large:
\begin{eqnarray}
| u (t^*)|_{ \dot{B}^{s}_{p, \infty  } }  \geq  |\overline{u} (t^*)|_{ \dot{B}^{s}_{p, \infty  } } - | w (t^*)|_{ \dot{B}^{s}_{p, \infty  } } \geq \ep^{-1}.
\end{eqnarray}

\end{itemize}

\end{proof}

When $s<0$, since Proposition \ref{prop:no_blowup} only provides the control of Sobolev norms with positive derivatives, it still remains to prove \eqref{eq:thm_Hs_NS_2}.

\begin{proof}[Proof of Theorem \ref{thm:Besov} for $s < 0$]

To show  \eqref{eq:thm_Hs_NS_2} when $s<0$, we use the mild formulation for the difference $w$ as in the proof of Proposition \ref{prop:no_blowup}. Since $w$ has zero initial data and is smooth on $[0,t^*]$, we have
\begin{align}\label{eq:aux_finalproof_1}
w(t) = \int_0^t  e^{(t-\tau)\Delta} \left( \overline{E} - \mathbb{P}\D( w\otimes u + \overline{u}\otimes w )\right) \, d\tau     .
\end{align}
We now take the $ \dot B^{s}_{p, \infty }$ norm on \eqref{eq:aux_finalproof_1}, and by the $\dot B^{s}_{p, \infty}$ boundedness of $ \mathbb{P} $ and of the heat semi-group there holds
\begin{align}\label{eq:aux_finalproof_2}
|w(t^*)|_{\dot B^{s}_{p,\infty}}  
&\lesssim       \underbrace{ t^* |  ( w\otimes u + \overline{u}\otimes w ) |_{L^\infty([0,t^*]; \dot B^{s+1}_{p, \infty })  }}_{:= \mathcal{N}} + \underbrace{t^* | \overline{E} |_{L^\infty([0,t^*]; \dot B^{s}_{p,\infty})  } }_{:= \mathcal{S}}     .
\end{align}

There are many ways to estimate  the terms $ \mathcal{N},\mathcal{S} $ on the right-hand side, and the non-optimal but simple one below suffices for our purpose.

\noindent
\textbf{\underline{Part 1: Estimate of $\mathcal{N}$}}

For the nonlinear terms, we first note that by Proposition \ref{prop:approximate} and Lemma  \ref{lemma:bootstrap} for any $k\in \NN$, $1 \leq {q} < 2$, there holds
\begin{equation}\label{eq:aux_finalproof_7}
\begin{aligned}
|  ( w\otimes u + \overline{u}\otimes w ) |_{L^\infty([0,t^*];  \dot W^{ k , {q} }  )  }  &  \lesssim  \sum_{0\leq i \leq k} | \nabla^{k-i }w |_{L^\infty_t  L^{ 2 }     } \Big(   | \nabla^i u |_{L^\infty_t   L^{  {q}}   )  }+ |  \nabla^i  \overline{u} |_{L^\infty_t L^{  {q}}     }   \Big)    \\
& \lesssim_{\ep, k, q}  \big(\mu^{-1}  \nu  \big)^{\frac{9}{10}}   \mu^{k  - 2s + \frac{2}{p } }  \nu^{\frac{1}{p}    }     \mu^{ \frac{2}{p }  - \frac{2}{ {q} }} \nu^{\frac{1}{p} - \frac{1}{ {q} }  }  
\end{aligned}
\end{equation}
where $q':= \frac{2q}{  2 - q } \in [  2 ,\infty)$. By interpolation \eqref{eq:aux_finalproof_7} holds for all $k\geq 0$.

Next,   consider the two cases $p \geq \frac{3}{2}$ and $ 1\leq p < \frac{3}{2}$. 

\textit{Case I: $ 1\leq p < \frac{3}{2}$.}

When  $ 1\leq p < \frac{3}{2}$, we have $s+1 >0$. Using Lemma \ref{lemma:besov_interpolation} with $s_1= 0$  and $s_2 = k_0+1$ together with \eqref{eq:aux_finalproof_7} yields
\begin{align}\label{eq:aux_finalproof_7a}
|  ( w\otimes u + \overline{u}\otimes w ) |_{L^\infty([0,t^*]; \dot B^{s+1}_{p,\infty })  }  &\lesssim_{\ep  }    \big(\mu^{-1}  \nu  \big)^{\frac{9}{10}}   \mu^{ 1  -  s + \frac{2}{p } }  \nu^{\frac{1}{p}    }      . 
\end{align}

\textit{Case II: $  p \geq  \frac{3}{2}$.}

When  $  p \geq  \frac{3}{2}$,   by the  embedding $\dot B^{s+3 - \frac{3}{p}   }_{ \frac{3}{2}, \infty } \hookrightarrow \dot B^{s+1}_{p, \infty }  $ and that $ s+3-  \frac{3}{p}   >0$, we can apply Lemma \ref{lemma:besov_interpolation} to interpolate \eqref{eq:aux_finalproof_7} with $q = \frac{3}{2}$ and obtain
\begin{align}\label{eq:aux_finalproof_7b}
|  ( w\otimes u + \overline{u}\otimes w ) |_{L^\infty([0,t^*]; \dot B^{s+1}_{p, \infty })  }  &\lesssim_{\ep }   |  ( w\otimes u + \overline{u}\otimes w ) |_{L^\infty([0,t^*]; \dot B^{s+3 - \frac{3}{p}   }_{ \frac{3}{2},  \infty } )  } \\
&\lesssim_{\ep }   \big(\mu^{-1}  \nu  \big)^{\frac{9}{10}}   \mu^{s+3 - \frac{3}{p}   - 2s + \frac{2}{p } }  \nu^{\frac{1}{p}    }     \mu^{ \frac{2}{p }  - \frac{4}{ 3 }} \nu^{\frac{1}{p} - \frac{2}{ 3 }  }   \\
&\lesssim_{\ep }   \big(\mu^{-1}  \nu  \big)^{\frac{9}{10}}   (\mu^{1  -      s + \frac{2}{p } }  \nu^{\frac{1}{p}    } )  \mu^{ 2 -       \frac{3}{p } }       \mu^{ \frac{2}{p }  - \frac{4}{ 3 }} \nu^{\frac{1}{p} - \frac{2}{ 3 }  }   \\
& \lesssim_{\ep }   \big(\mu^{-1}  \nu  \big)^{\frac{9}{10}}   \mu^{1  -  s + \frac{2}{p } }  \nu^{\frac{1}{p}    }   \Big( \mu^{-1} \nu \Big)^{ - \frac{2}{3} + \frac{1}{p}   }     
\end{align}
which implies 
\begin{align}
t^* |  ( w\otimes u + \overline{u}\otimes w ) |_{L^\infty([0,t^*]; \dot B^{s+1}_{p, \infty })  }  & \lesssim_{\ep } \big(\mu^{-1}  \nu  \big)^{  \frac{9}{10} -\frac{2}{3}  + \frac{1}{p} }  \\
& \lesssim_{\ep } \big(\mu^{-1}  \nu  \big)^{  \frac{1}{ 5}    }.\label{eq:aux_finalproof_7c}
\end{align}

\noindent
\textbf{\underline{Part 2: Estimate of $\mathcal{S}$}}

Finally, we estimate the source term. Consider two cases: $ p=1$ and $p>1$. 

\textit{Case I: $  p = 1 $.}

If $p =1$, then $ s>0$ and we can proceed using Lemma \ref{lemma:besov_interpolation} with $s_1= 0$  and $s_2 = k_0$ together with Proposition \ref{prop:approximate}  to obtain that
\begin{align}\label{eq:aux_finalproof_8}
| \overline{E} |_{L^\infty([0,t^*]; \dot B^{ s  }_{p   , \infty })   }  & \lesssim_{\ep }     (\mu^{-1} \nu)   (\mu^{1+\frac{ 2}{p}-  s}    \nu^\frac{1}{p}  )   
\end{align}

\textit{Case I: $  p > 1 $.}

If $p>1$, we choose $ 0< \delta < 1/5$ small such that $ s+ 3 - \frac{3}{p} -\delta >0 $ and   $p_\delta: = \frac{3}{3 -\delta} < p$.   By the embedding $\dot B^{ s+ 3 - \frac{3}{p} -\delta  }_{p_\delta  , \infty } \hookrightarrow \dot B^{ s     }_{p   , \infty } $, Lemma \ref{lemma:besov_interpolation}, and Proposition \ref{prop:approximate} we get 
\begin{align}\label{eq:aux_finalproof_8a}
| \overline{E} |_{L^\infty([0,t^*]; \dot B^{ s+ 3 - \frac{3}{p} -\delta  }_{p_\delta  , \infty })   }  & \lesssim_{\ep }   \mu^{ s+ 3 - \frac{3}{p} -\delta - s }  (\mu^{-1} \nu)   (\mu^{1+\frac{ 2}{p}-  s}    \nu^\frac{1}{p}  )  \mu^{ \frac{ 2}{p}  - \frac{2}{p_\delta} } \nu^{\frac{1}{p} - \frac{1}{p_\delta} } \\
& \lesssim_{\ep }  (\mu^{-1} \nu)    \mu^{   3 - \frac{3}{p} -\delta   }   (\mu^{1+\frac{ 2}{p}-  s}    \nu^\frac{1}{p}  )  \mu^{ \frac{2}{p} - 2  + \frac{2\delta }{3} } \nu^{\frac{1}{p} - 1 + \frac{ \delta }{3} }  \\
& \lesssim_{\ep }  (\mu^{-1} \nu)^{\frac{1}{p}   + \frac{ \delta }{3} }    (\mu^{1+\frac{ 2}{p}-  s}    \nu^\frac{1}{p}  )    .
\end{align}

Therefore, in either case by \eqref{eq:aux_finalproof_8} and \eqref{eq:aux_finalproof_8a} we have
\begin{align}\label{eq:aux_finalproof_8b}
t^* | \overline{E} |_{L^\infty([0,t^*]; \dot B^{ s+ 3 - \frac{3}{p} -\delta  }_{p_\delta  , \infty })   }    
& \lesssim_{\ep } (\mu^{-1} \nu)^{ \frac{1}{p} + \frac{\delta}{3}}     \leq (\mu^{-1} \nu)^{   \frac{\delta}{3}}  .
\end{align}

\noindent
\textbf{\underline{Part 4: Conclusion}}

Collecting    \eqref{eq:aux_finalproof_7c} for $\mathcal{N}$  and  \eqref{eq:aux_finalproof_8b} for $\mathcal{S}$ we obtain
\begin{align}\label{eq:aux_finalproof_20}
|w(t^*)|_{\dot B^{s}_{p, \infty }} & \lesssim_{\ep }    (\mu^{-1} \nu)^{   \frac{\delta}{3}}  
\end{align}
for some small $ 0< \delta < 1/5$.

We still have the freedom to increase the value of $\mu$, and we can choose $\mu$ sufficiently large such that in \eqref{eq:aux_finalproof_20} we have
$$
|w(t^*)|_{\dot B^{s}_{p, \infty }}  \leq \ep.
$$
Thanks to Proposition \ref{prop:approximate_inflation_s<0}, this completes the proof since 
$$
|u(t^*)|_{\dot B^{s}_{p, \infty }} \geq |\overline{u}(t^*)|_{\dot B^{s}_{p,\infty}} -|w(t^*)|_{\dot B^{s}_{p, \infty }} \geq C \ep^{-2} - \ep \geq \ep^{-1} 
$$ provided $\ep $ is small.

\end{proof}

\subsection{Proof of Theorem \ref{thm:growth}}

Finally we prove Theorem \ref{thm:growth}.

To avoid visual confusion with $w= u -  \overline{u} $, we denote the vorticity by $\nabla \times u$.

Since any supercritical Sobolev/Besov spaces embeds into $\dot W^{2-\delta,1}$ when $\delta>0$ is small enough, we apply Theorem \ref{thm:Besov} with $s$ close to $2$ and $p=1$ to obtain the solution with initial data that is  $\ep$-small in the prescribed supercritical Sobolev/Besov norm.

To obtain a lower bound for $ |\nabla \times u (t^*)|_{L^\infty}  $ we consider
\begin{equation}\label{eq:proof_thm:growth}
|\nabla \times u  (t^*) |_{L^\infty} \geq  |  \nabla \times \overline{u} (t^*)  |_{L^\infty}  -  |\nabla   w(t^*)|_{L^\infty}.
\end{equation}

The estimate of  $|\nabla \times \overline{u}(t^*) |_{L^\infty}$ was proved in Lemma \ref{lemma:vorticity_s>0}:
\begin{equation}\label{eq:proof_thm:growth_2}
|\nabla \times \overline{u}(t^*) |_{L^\infty} \geq C \ep^{-2 } \mu^{1-s}( \mu^2 \nu^1).
\end{equation}
where we recall that $p=1$ has been chosen.

Therefore, by Proposition \ref{prop:no_blowup}, \eqref{eq:proof_thm:growth}, and \eqref{eq:proof_thm:growth_2}, for all sufficiently large $\mu$  such a solution $u$ satisfies
\begin{equation}\label{eq:proof_thm:growth_3}
|\nabla \times u(t^*) |_{L^\infty} \geq  C\ep^{-2} \mu^{1-s} ( \mu^2 \nu^1) .
\end{equation}

The upper bound $ |\nabla\times u_0|_{L^\infty}  $ follows from  \eqref{eq:aux_prop:approximate_inflation_s>0_1}, yielding
\begin{equation}\label{eq:proof_thm:growth_4}
|\nabla \times u_0 |_{L^\infty} \leq   |\nabla  u_0 |_{L^\infty} \leq C \ep^{2} \mu^{1-s} ( \mu^2 \nu^1) .
\end{equation}

It follows from \eqref{eq:proof_thm:growth_3} and \eqref{eq:proof_thm:growth_4} that
\begin{equation}
\frac{|\nabla \times u(t^*) |_{L^\infty} }{|\nabla \times  u_0 |_{L^\infty}} \geq C\ep^{-4} \geq M
\end{equation}
provided $\ep>0$ is sufficiently small.

\bibliographystyle{alpha}
\bibliography{NS_ill_2024}

\end{document}